\documentclass{article}
\usepackage{latexsym}
\usepackage{amssymb,amsmath, amscd}
\usepackage{setspace}
\usepackage{fancyhdr}
\usepackage{amsthm}
\usepackage{graphicx}
\usepackage{tensor}
\usepackage{natbib} 
\usepackage{tikz}
\usepackage{mathdots}

\newtheorem{theorem}{Theorem}[section]
\newtheorem{corollary}{Corollary}[section]
\newtheorem{prop}{Proposition}[section]
\newtheorem{lemma}{Lemma}[section]

\theoremstyle{definition}
\newtheorem{example1}{Example}
\newtheorem*{definition}{Definition}
\theoremstyle{remark}

\newcommand\xqed[1]{%
  \leavevmode\unskip\penalty9999 \hbox{}\nobreak\hfill
  \quad\hbox{#1}}
\newcommand\corner{\xqed{$\lrcorner$}}

\tikzset{node distance=1.6cm, auto}

\newenvironment{example}[0]{\begin{example1}}{\corner\end{example1}}

\usepackage{breakcites}

\makeatletter
\renewcommand\@biblabel[1]{}
\makeatother

\title{Morita Equivalence\thanks{The authors can be reached at thomaswb@princeton.edu and hhalvors@princeton.edu. We would like to thank Dimitris Tsementzis, Jim Weatherall, and JB Manchak for helpful comments and discussion.} 
}

\author{Thomas William Barrett\\ 
Hans Halvorson
}

\date{}

\begin{document}
\maketitle

\begin{abstract}
%
%
Logicians and philosophers of science have proposed various formal criteria for theoretical equivalence. In this paper, we examine two such proposals: definitional equivalence and categorical equivalence. In order to show precisely how these two well-known criteria are related to one another, we investigate an intermediate criterion called Morita equivalence.
\end{abstract}


\section{Introduction}
Many theories admit different formulations, and these formulations often bear interesting relationships to one another. One relationship that has received significant attention from logicians and philosophers of science is theoretical equivalence.\footnote{See \cite{quine1975}, \cite{sklar1982}, \cite{halvorson2012, halvorson2013, halvorson2015}, \cite{glymour2013}, \cite{vanfraassen2014}, and \cite{coffey2014} for discussion of theoretical equivalence in philosophy of science.} In this paper we will examine two formal criteria for theoretical equivalence. The first criterion, called \textit{definitional equivalence}, has been known to logicians since the middle of the twentieth century.\footnote{\cite{artigue1978} and \cite{debouvere1965} attribute the concept of definitional equivalence to \cite{montague1957}. Definitional equivalence was certainly familiar to logicians by the late 1960s, as is evident from the work of \cite{debouvere1965}, \cite{shoenfield1967}, and \cite{kanger1968}.} It was introduced into philosophy of science by \cite{glymour1970, glymour1977, glymour1980}. The second criterion is called \textit{categorical equivalence}. It was first described by \cite{eilenbergmaclane1942, eilenbergmaclane1945}, but was only recently introduced into philosophy of science by \cite{halvorson2012, halvorson2015} and \cite{weatherallunpublished}.

In order to illustrate the relationship between these two criteria, we will consider a third criterion for theoretical equivalence called \textit{Morita equivalence.} We will show that these three criteria form the following hierarchy, where the arrows in the figure mean ``implies.''
\begin{center}
\includegraphics{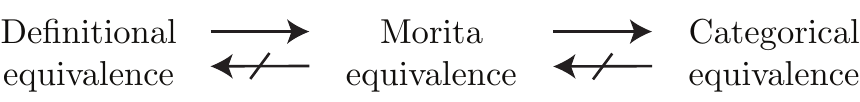}
\end{center}
Our discussion will allow us to evaluate definitional equivalence against categorical equivalence. Indeed, it will demonstrate a precise sense in which definitional equivalence is too strict a criterion for theoretical equivalence, while categorical equivalence is too liberal. There are theories that are not definitionally equivalent that one nonetheless has good reason to consider equivalent. And on the other hand, there are theories that are categorically equivalent that one has good reason to consider \textit{in}equivalent.

\section{Many-sorted logic}

All of these criteria for theoretical equivalence are most naturally understood in the framework of first-order many-sorted logic. We begin with some preliminaries about this framework.\footnote{Our notation follows \cite{hodges2008}. We present the more general case of many-sorted logic, however, while Hodges only presents single-sorted logic.}

\subsection{Syntax}

A \textbf{signature} $\Sigma$ is a set of sort symbols, predicate symbols, function symbols, and constant symbols. $\Sigma$ must have at least one sort symbol. Each predicate symbol $p\in\Sigma$ has an \textbf{arity} $\sigma_1\times\ldots\times\sigma_n$, where $\sigma_1,\ldots, \sigma_n\in\Sigma$ are (not necessarily distinct) sort symbols. Likewise, each function symbol $f\in\Sigma$ has an \textbf{arity} $\sigma_1\times\ldots\times\sigma_n\rightarrow\sigma$, where $\sigma_1,\ldots, \sigma_n,\sigma\in\Sigma$ are again (not necessarily distinct) sort symbols. Lastly, each constant symbol $c\in\Sigma$ is assigned a sort $\sigma\in\Sigma$. In addition to the elements of $\Sigma$ we also have a stock of variables. We use the letters $x$, $y$, and $z$ to denote these variables, adding subscripts when necessary. Each variable has a sort $\sigma\in\Sigma$.

A \textbf{$\mathbf{\Sigma}$-term} can be thought of as a ``naming expression'' in the signature $\Sigma$. Each $\Sigma$-term has a sort $\sigma\in\Sigma$. The $\Sigma$-terms of sort $\sigma$ are recursively defined as follows. Every variable of sort $\sigma$ is a $\Sigma$-term of sort $\sigma$, and every constant symbol $c\in\Sigma$ of sort $\sigma$ is also a $\Sigma$-term of sort $\sigma$. Furthermore, if $f\in\Sigma$ is a function symbol with arity $\sigma_1\times\ldots\times\sigma_n\rightarrow\sigma$ and $t_1,\ldots, t_n$ are $\Sigma$-terms of sorts $\sigma_1, \ldots, \sigma_n$, then $f(t_1,\ldots, t_n)$ is a $\Sigma$-term of sort $\sigma$. We will use the notation $t(x_1,\ldots, x_n)$ to denote a $\Sigma$-term in which all of the variables that appear in $t$ are in the sequence $x_1,\ldots, x_n$, but we leave open the possibility that some of the $x_i$ do not appear in the term $t$. 

A \textbf{$\mathbf{\Sigma}$-atom} is an expression either of the form $s(x_1,\ldots, x_n)=t(x_1,\ldots, x_n)$, where $s$ and $t$ are $\Sigma$-terms of the same sort $\sigma\in\Sigma$, or of the form $p(t_1,\ldots, t_n)$, where $t_1,\ldots, t_n$ are $\Sigma$-terms of sorts $\sigma_1, \ldots, \sigma_n$ and $p\in\Sigma$ is a predicate of arity $\sigma_1\times\ldots\times\sigma_n$. The \textbf{$\mathbf{\Sigma}$-formulas} are then defined recursively as follows.
\begin{itemize}
\item Every $\Sigma$-atom is a $\Sigma$-formula.
\item If $\phi$ is a $\Sigma$-formula, then $\lnot\phi$ is a $\Sigma$-formula.
\item If $\phi$ and $\psi$ are $\Sigma$-formulas, then $\phi\rightarrow\psi$, $\phi\land\psi$, $\phi\lor\psi$ and $\phi\leftrightarrow\psi$ are $\Sigma$-formulas.
\item If $\phi$ is a $\Sigma$-formula and $x$ is a variable of sort $\sigma\in\Sigma$, then $\forall_\sigma x\phi$ and $\exists_\sigma x\phi$ are $\Sigma$-formulas.
\end{itemize}
In addition to the above formulas, we will use the notation $\exists_{\sigma=1}y \phi(x_1,\ldots, x_n, y)$ to abbreviate the formula $\exists_\sigma y(\phi(x_1,\ldots, x_n, y)\land\forall_\sigma z(\phi(x_1,\ldots, x_n,z)\rightarrow y=z))$. As above, the notation $\phi(x_1,\ldots, x_n)$ will denote a $\Sigma$-formula $\phi$ in which all of the free variables appearing in $\phi$ are in the sequence $x_1,\ldots, x_n$, but we again leave open the possibility that some of the $x_i$ do not appear as free variables in $\phi$. A $\mathbf{\Sigma}$\textbf{-sentence} is a $\Sigma$-formula that has no free variables.

\subsection{Semantics}

A \textbf{$\mathbf{\Sigma}$-structure} $A$ is an ``interpretation'' of the symbols in $\Sigma$. In particular, $A$ satisfies the following conditions.
\begin{itemize}
\item Every sort symbol $\sigma\in\Sigma$ is assigned a nonempty set $A_\sigma$. The sets $A_\sigma$ are required to be pairwise disjoint.
\item Every predicate symbol $p\in\Sigma$ of arity $\sigma_1\times\ldots\times\sigma_n$ is interpreted as a subset $p^A\subset A_{\sigma_1}\times\ldots\times A_{\sigma_n}$.
\item Every function symbol $f\in\Sigma$ of arity $\sigma_1\times\ldots\times\sigma_n\rightarrow\sigma$ is interpreted as a function $f^A:A_{\sigma_1}\times\ldots\times A_{\sigma_n}\rightarrow A_\sigma$.
\item Every constant symbol $c\in\Sigma$ of sort $\sigma\in\Sigma$ is interpreted as an element $c^A\in A_\sigma$.
\end{itemize}
Given a $\Sigma$-structure $A$, we will often refer to an element $a\in A_\sigma$ as ``an element of sort $\sigma$.''

Let $A$ be a $\Sigma$-structure with $a_1,\ldots, a_n\in A$ elements of sorts $\sigma_1,\ldots, \sigma_n$. We let $t(x_1,\ldots, x_n)$ be a $\Sigma$-term of sort $\sigma$, with $x_1,\ldots, x_n$ variables of sorts $\sigma_1,\ldots, \sigma_n$, and we recursively define the element $t^A[a_1,\ldots, a_n]\in A_\sigma$. If $t$ is the variable $x_i$, then $t^A[a_1,\ldots, a_n]=a_i$, and if $t$ is the constant symbol $c\in\Sigma$, then $t^A[a_1,\ldots, a_n]=c^A$. Furthermore, if $t$ is of the form $f(t_1,\ldots, t_m)$ where each $t_i$ is a $\Sigma$-term of sort $\tau_i\in\Sigma$ and $f\in\Sigma$ is a function symbol of arity $\tau_1\times\ldots\times\tau_m\rightarrow\sigma$, then
$$
t^A[a_1,\ldots, a_n]=f^A\big(t_1^A[a_1,\ldots, a_n],\ldots, t_m^A[a_1,\ldots, a_n]\big)
$$
One can think of the element $t^A[a_1,\ldots, a_n]\in A_\sigma$ as the element of the $\Sigma$-structure $A$ that is denoted by the $\Sigma$-term $t(x_1,\ldots, x_n)$ when $a_1,\ldots, a_n$ are substituted for the variables $x_1,\ldots, x_n$.

Our next aim is to define when a sequence of elements $a_1,\ldots, a_n\in A$ \textbf{satisfy} a $\Sigma$-formula $\phi(x_1,\ldots, x_n)$ in the $\Sigma$-structure $A$. When this is the case we write $A\vDash\phi[a_1,\ldots, a_n]$. We begin by considering $\Sigma$-atoms. Let $\phi(x_1,\ldots, x_n)$ be a $\Sigma$-atom with $x_1,\ldots, x_n$ variables of sorts $\sigma_1,\ldots, \sigma_n$ and let $a_1,\ldots, a_n\in A$ be elements of sorts $\sigma_1,\ldots, \sigma_n$. There are two cases to consider. First, if $\phi(x_1,\ldots, x_n)$ is the formula $s(x_1,\ldots, x_n)=t(x_1,\ldots, x_n)$, where $s$ and $t$ are $\Sigma$-terms of sort $\sigma$, then $A\vDash\phi[a_1,\ldots, a_n]$ if and only if
$$
s^A[a_1,\ldots, a_n]=t^A[a_1,\ldots, a_n]
$$
Second, if $\phi(x_1,\ldots, x_n)$ is the formula $p(t_1,\ldots, t_m)$, where each $t_i$ is a $\Sigma$-term of sort $\tau_i$ and $p\in\Sigma$ is a predicate symbol of arity $\tau_1\times\ldots\times \tau_m$, then $A\vDash\phi[a_1,\ldots, a_n]$ if and only if
$$
\big(t_1^A[a_1,\ldots, a_n], \ldots, t_m^A[a_1,\ldots, a_n]\big)\in p^A
$$
This definition is extended to all $\Sigma$-formulas in the following standard way.
\begin{itemize}
\item $A\vDash\lnot\phi[a_1,\ldots, a_n]$ if and only if it is not the case that $A\vDash\phi[a_1,\ldots, a_n]$.

\item $A\vDash \phi\land\psi[a_1,\ldots, a_n]$ if and only if $A\vDash\phi[a_1,\ldots, a_n]$ and $A\vDash\psi[a_1,\ldots, a_n]$. The cases of $\lor$, $\rightarrow$, and $\leftrightarrow$ are defined analogously.

\item Suppose that $\phi(x_1,\ldots, x_n)$ is $\forall_\sigma y \psi(x_1,\ldots, x_n, y)$, where $\sigma\in\Sigma$ is a sort symbol. Then $A\vDash\phi[a_1,\ldots, a_n]$ if and only if $A\vDash\psi[a_1,\ldots, a_n, b]$ for every element $b\in A_\sigma$. The case of $\exists_\sigma$ is defined analogously. 
\end{itemize}
If $\phi$ is a $\Sigma$-sentence, then $A\vDash\phi$ just in case $A\vDash\phi[]$, i.e.~the empty sequence satisfies $\phi$ in $A$.

\subsection{Relationships between structures}

There are different relationships that $\Sigma$-structures can bear to one another. An \textbf{isomorphism} $h:A\rightarrow B$ between $\Sigma$-structures $A$ and $B$ is a family of bijections $h_\sigma:A_\sigma\rightarrow B_\sigma$ for each sort symbol $\sigma\in\Sigma$ that satisfies the following conditions.
\begin{itemize}
\item For every predicate symbol $p\in\Sigma$ of arity $\sigma_1\times\ldots\times\sigma_n$ and all elements $a_1,\ldots, a_n\in A$ of sorts $\sigma_1,\ldots,\sigma_n$, $(a_1,\ldots, a_n)\in p^A$ if and only if $(h_{\sigma_1}(a_1),\ldots, h_{\sigma_n}(a_n))\in p^B$.
\item For every function symbol $f\in\Sigma$ of arity $\sigma_1\times\ldots\times\sigma_n\rightarrow\sigma$ and all elements $a_1,\ldots, a_n\in A$ of sorts $\sigma_1,\ldots,\sigma_n$,
$$
h_\sigma\big(f^A(a_1,\ldots, a_n)\big)=f^B\big(h_{\sigma_1}(a_1),\ldots, h_{\sigma_n}(a_n)\big)
$$
\item For every constant symbol $c\in\Sigma$ of sort $\sigma$, $h_\sigma(c^A)=c^B$.
\end{itemize}
When there is an isomorphism $h:A\rightarrow B$ one says that $A$ and $B$ are \textbf{isomorphic} and writes $A\cong B$.

There is another important relationship that $\Sigma$-structures can bear to one another. An \textbf{elementary embedding} $h:A\rightarrow B$ between $\Sigma$-structures $A$ and $B$ is a family of maps $h_\sigma:A_\sigma\rightarrow B_\sigma$ for each sort symbol $\sigma\in\Sigma$ that satisfies 
$$
A\vDash\phi[a_1,\ldots, a_n]\text{ if and only if } B\vDash\phi[h_{\sigma_1}(a_1),\ldots, h_{\sigma_n}(a_n)]
$$
for all $\Sigma$-formulas $\phi(x_1,\ldots, x_n)$ 
and elements $a_1,\ldots, a_n\in A$ of sorts $\sigma_1,\ldots, \sigma_n$. Given an isomorphism or elementary embedding $h:A\rightarrow B$, we will often use the notation $h(a_1,\ldots, a_n)$ to denote the sequence of elements $h_{\sigma_1}(a_1),\ldots, h_{\sigma_n}(a_n)$. Every isomorphism is an elementary embedding, but in general the converse does not hold.

There is an important relationship that can hold between structures of different signatures. Let $\Sigma\subset\Sigma^+$ be signatures and suppose that $A$ is a $\Sigma^+$-structure. One obtains a $\Sigma$-structure $A|_\Sigma$ by ``forgetting'' the interpretations of symbols in $\Sigma^+-\Sigma$. We call $A|_\Sigma$ the \textbf{reduct} of $A$ to the signature $\Sigma$, and we call $A$ an \textbf{expansion} of $A|_\Sigma$ to the signature $\Sigma^+$. Note that in general a $\Sigma$-structure will have more than one expansion to the signature $\Sigma^+$.

We can now discuss first-order theories in many-sorted logic. A \textbf{$\mathbf{\Sigma}$-theory} $T$ is a set of $\Sigma$-sentences. The sentences $\phi\in T$ are called the axioms of $T$. A $\Sigma$-structure $M$ is a \textbf{model} of a $\Sigma$-theory $T$ if $M\vDash\phi$ for all $\phi\in T$. A theory \textbf{$T$ entails} a sentence $\phi$, written $T\vDash\phi$, if $M\vDash\phi$ for every model $M$ of $T$. 

We begin our discussion of theoretical equivalence with the following preliminary criterion.
\begin{definition}
Theories $T_1$ and $T_2$ are \textbf{logically equivalent} if they have the same class of models.
\end{definition}
One can easily verify that $T_1$ and $T_2$ are logically equivalent if and only if $\{\phi: T_1\vDash\phi\}=\{\psi: T_2\vDash\psi\}$.

\section{Definitional equivalence}

Logical equivalence is a particularly strict criterion for theoretical equivalence. Indeed, theories can only be logically equivalent if they are formulated in the same signature. There are many cases, however, of theories in different signatures that are nonetheless intuitively equivalent. For example, the theory of groups can be formulated in a signature with a binary operation $\cdot$ and a constant symbol $e$, or it can be formulated in a signature with a binary operation $\cdot$ and a unary function $-1$ \citep{barrett2014a}. Similarly, the theory of linear orders can be formulated in a signature with the binary relation $<$, or it can be formulated in a signature with the binary relation $\leq$. Since logical equivalence does not capture any sense in which these theories are equivalent, logicians and philosophers of science have proposed more general criteria for theoretical equivalence. 

One such criterion is \textit{definitional equivalence}. This criterion is well known among logicians, and many results about it have been proven.\footnote{For example, see \cite{debouvere1965}, \cite{kanger1968}, \cite{pinter1978}, \cite{pelletier2003}, \cite{andreka2005}, \cite{friedman2014}, and \cite{barrett2014a}.} The basic idea behind definitional equivalence is simple. Theories $T_1$ and $T_2$ are definitionally equivalent if $T_1$ can define all of the symbols that $T_2$ uses, and in a compatible way, $T_2$ can define all of the symbols that $T_1$ uses. In order to state this criterion precisely, we need to do some work.

\subsection{Definitional extensions}

We first need to formalize the concept of a definition. Let $\Sigma\subset\Sigma^+$ be signatures and let $p\in\Sigma^+-\Sigma$ be a predicate symbol of arity $\sigma_1\times\ldots\times\sigma_n$. An \textbf{explicit definition of $p$ in terms of $\mathbf{\Sigma}$} is a $\Sigma^+$-sentence of the form
$$
\forall_{\sigma_1}x_1\ldots\forall_{\sigma_n}x_n\big(p(x_1,\ldots,x_n)\leftrightarrow\phi(x_1,\ldots, x_n)\big)
$$
where $\phi(x_1,\ldots, x_n)$ is a $\Sigma$-formula. Note that an explicit definition of $p$ in terms of $\Sigma$ can only exist if $\sigma_1,\ldots, \sigma_n\in \Sigma$. An explicit definition of a function symbol $f\in\Sigma^+-\Sigma$ of arity $\sigma_1\times\ldots\times\sigma_n\rightarrow\sigma$ is a $\Sigma^+$-sentence of the form
\begin{equation}
\label{functiondefinition}
\forall_{\sigma_1}x_1\ldots\forall_{\sigma_n}x_n\forall_\sigma y\big( f(x_1,\ldots, x_n)=y\leftrightarrow\phi(x_1,\ldots, x_n, y)\big)
\end{equation}
and an explicit definition of a constant symbol $c\in\Sigma^+-\Sigma$ of sort $\sigma$ is a $\Sigma^+$-sentence of the form 
\begin{equation}
\label{constantdefinition}
\forall_\sigma y\big(y=c\leftrightarrow\psi(y)\big)
\end{equation}
where $\phi(x_1,\ldots, x_n, y)$ and $\psi(y)$ are both $\Sigma$-formulas. Note again that these explicit definitions of $f$ and $c$ can only exist if $\sigma_1,\ldots, \sigma_n, \sigma\in\Sigma$.

Although they are $\Sigma^+$-sentences, \eqref{functiondefinition} and \eqref{constantdefinition} have consequences in the signature $\Sigma$. In particular, \eqref{functiondefinition} and \eqref{constantdefinition} imply the following sentences, respectively:
\begin{align*}
&\forall_{\sigma_1}x_1\ldots\forall_{\sigma_n}x_n\exists_{\sigma=1}y\phi(x_1,\ldots, x_n, y)\\
&\exists_{\sigma=1}y\psi(y)
\end{align*}
These two sentences are called the \textbf{admissibility conditions} for the explicit definitions \eqref{functiondefinition} and \eqref{constantdefinition}. 

A \textbf{definitional extension} of a $\Sigma$-theory $T$ to the signature $\Sigma^+$ is a theory
$$
T^+=T\cup\{\delta_s:s\in\Sigma^+-\Sigma\}
$$
that satisfies the following two conditions. First, for each symbol $s\in\Sigma^+-\Sigma$ the sentence $\delta_s$ is an explicit definition of $s$ in terms of $\Sigma$, and second, if $s$ is a constant symbol or a function symbol and $\alpha_s$ is the admissibility condition for $\delta_s$, then $T\vDash\alpha_s$. 

\subsection{Three results}

A definitional extension of a theory ``says no more'' than the original theory. There are a number of ways to make this idea precise. Of particular interest to us will be the following three. The reader is encouraged to consult \citet[p.~58--62]{hodges2008} for proofs of these results.

The first result captures a sense in which the models of a definitional extension $T^+$ are ``determined'' by the models of the original theory $T$. In order to specify a model of $T^+$, one needs to interpret all of the symbols in $\Sigma^+$. The interpretation of the symbols in $\Sigma^+-\Sigma$, however, ``comes for free'' given an interpretation of the symbols in $\Sigma$.
\begin{theorem}
Let $\Sigma\subset\Sigma^+$ be signatures and $T$ a $\Sigma$-theory. If $T^+$ is a definitional extension of $T$ to $\Sigma^+$, then every model $M$ of $T$ has a unique expansion $M^+$ that is a model of $T^+$.
\end{theorem}

Theorem 3.1 provides a semantic sense in which a definitional extension $T^+$ ``says no more'' than the original theory $T$. The models of $T^+$ are completely determined by the models of $T$.

In order to state the second result, we need to introduce some terminology. Let $\Sigma\subset\Sigma^+$ be signatures. A $\Sigma^+$-theory $T^+$ is an \textbf{extension} of a $\Sigma$-theory $T$ if $T\vDash \phi$ implies that $T^+\vDash\phi$ for every $\Sigma$-sentence $\phi$. A $\Sigma^+$-theory $T^+$ is a \textbf{conservative extension} of a $\Sigma$-theory $T$ if $T\vDash\phi$ if and only if $T^+\vDash \phi$ for every $\Sigma$-sentence $\phi$. All conservative extensions are extensions, but in general the converse does not hold. We have the following simple result about definitional extensions. 

\begin{theorem}
If $T^+$ is a definitional extension of $T$, then $T^+$ is a conservative extension of $T$.
\end{theorem}
If $T^+$ is a conservative extension of $T$, then $T^+$ entails precisely the same $\Sigma$-sentences as $T$. Theorem 3.2 therefore shows that a definitional extension $T^+$ ``says no more'' in the signature $\Sigma$ than the original theory $T$ does.

The third result shows something stronger. If $T^+$ is a definitional extension of $T$, then every $\Sigma^+$-formula $\phi(x_1,\ldots, x_n)$ can be ``translated'' into an equivalent $\Sigma$-formula $\phi^*(x_1,\ldots, x_n)$. The theory $T^+$ might use some new language that $T$ did not use, but everything that $T^+$ says with this new language can be ``translated'' back into the old language of $T$. This result captures another robust sense in which the theory $T^+$ ``says no more'' than the theory $T$. 

\begin{theorem}
Let $\Sigma\subset\Sigma^+$ be signatures and $T$ a $\Sigma$-theory. If $T^+$ is a definitional extension of $T$ to $\Sigma^+$ then for every $\Sigma^+$-formula $\phi(x_1,\ldots, x_n)$ there is a $\Sigma$-formula $\phi^*(x_1,\ldots, x_n)$ such that 
$
T^+\vDash\forall_{\sigma_1}x_1\ldots\forall_{\sigma_n}x_n(\phi(x_1,\ldots, x_n)\leftrightarrow\phi^*(x_1,\ldots, x_n))
$.
\end{theorem}

These results capture three different senses in which a definitional extension has the same expressive power as the original theory. With this in mind, we have the resources necessary to state definitional equivalence.

\begin{definition}
Let $T_1$ be a $\Sigma_1$-theory and $T_2$ be a $\Sigma_2$-theory. $T_1$ and $T_2$ are \textbf{definitionally equivalent} if there are theories $T_1^+$ and $T_2^+$ that satisfy the following three conditions:
\begin{itemize}
\item $T_1^+$ is a definitional extension of $T_1$,
\item $T_2^+$ is a definitional extension of $T_2$,
\item $T_1^+$ and $T_2^+$ are logically equivalent $\Sigma_1\cup\Sigma_2$-theories.
\end{itemize}
\end{definition}
One often says that $T_1$ and $T_2$ are definitionally equivalent if they have a ``common definitional extension.'' Theorems 3.1, 3.2, and 3.3 demonstrate a robust sense in which theories with a common definitional extension ``say the same thing,'' even though they might be formulated in different signatures. 

One trivially sees that if two theories are logically equivalent, then they are definitionally equivalent. But there are many examples of theories that are definitionally equivalent and not logically equivalent. The theory of groups formulated in the signature $\{\cdot, e\}$ is definitionally equivalent to the theory of groups formulated in the signature $\{\cdot, -1\}$. And likewise, the theory of linear orders formulated in the signature $\{<\}$ is definitionally equivalent to the theory of linear orders formulated in the signature $\{\leq\}$. Definitional equivalence is therefore a weaker criterion for theoretical equivalence than logical equivalence. It is capable of capturing a sense in which theories formulated in different signatures might nonetheless be equivalent.

\section{Morita equivalence}

Definitional equivalence, however, is \textit{in}capable of capturing any sense in which theories formulated with different sorts might be equivalent.  We have provided no way of defining new sort symbols. One can therefore easily verify that if $T_1$ and $T_2$ are definitionally equivalent, then it must be that $\Sigma_1$ and $\Sigma_2$ have the same sort symbols. There are many theories with different sort symbols, however, that one has good reason to consider equivalent.

One particularly famous example of this is Euclidean geometry. It can be formulated with only a sort of ``points'' \citep{tarski1959}, with only a sort of ``lines'' \citep{schwabhauser1975}, or with both a sort of ``points'' and a sort of ``lines'' \citep{hilbert1930}.\footnote{\cite{szczerba1977} and \citet[Proposition 4.59, Proposition 4.89]{schwabhauser1983} discuss the relationships between these formulations.} Category theory can also be formulated using different sorts. The standard formulation uses both a sort of ``objects'' and a sort of ``arrows'' \citep{eilenbergmaclane1942, eilenbergmaclane1945}. But it is well known that category theory can instead be formulated using only a sort of ``arrows'' \citep{maclane1948}.\footnote{\citet[p.~5]{freyd1964} and \citet[p.~9]{cwm} also describe this alternative formulation.} Since these formulations use different sort symbols, definitional equivalence does not capture any sense in which they are equivalent. 

In addition to these two famous examples, we have the following simple example. 
\begin{example}
\label{motivating example}
Let $\Sigma_1=\{\sigma_1, p, q\}$ and $\Sigma_2=\{\sigma_2, \sigma_3\}$ be signatures with $\sigma_1, \sigma_2$, and $\sigma_3$ sort symbols and $p$ and $q$ predicate symbols of arity $\sigma_1$. Consider the $\Sigma_1$-theory 
\begin{align*}
T_1=\big\{&\exists_{\sigma_1} x_1 p(x_1), \exists_{\sigma_1} x_2q(x_2), \forall_{\sigma_1} x_1(p(x_1)\lor q(x_1)),\\
 &\forall_{\sigma_1} x_1\lnot(p(x_1)\land q(x_1))\big\}
\end{align*}
and the $\Sigma_2$-theory $T_2=\emptyset$. Since the signatures $\Sigma_1$ and $\Sigma_2$ have different sort symbols, $T_1$ and $T_2$ are not definitionally equivalent.
\end{example}

Even though $T_1$ and $T_2$ are not definitionally equivalent, one still has good reason to consider them equivalent. The theory $T_1$ partitions everything into the things that are $p$ and the things that are $q$. Similarly, the theory $T_2$ partitions everything into the things of sort $\sigma_1$ and the things of sort $\sigma_2$. Both $T_1$ and $T_2$ say ``there are two kinds of things.'' The only difference between them is that $T_1$ uses predicates to say this, while $T_2$ uses sorts. 

These examples all show that definitional equivalence does not capture the sense in which some theories are equivalent. If one wants to capture this sense, one needs a more general criterion for theoretical equivalence than definitional equivalence. Our aim here is to introduce one such criterion. We will call it \textit{Morita equivalence}.\footnote{This criterion is already familiar in certain circles of logicians. See \cite{andreka2008}. The name ``Morita equivalence'' descends from Kiiti Morita's work on rings with equivalent categories of modules.  Two rings $R$ and $S$ are called \emph{Morita equivalent} just in case there is an equivalence $\text{Mod}(R)\cong \text{Mod}(S)$ between their categories of modules.  The notion was generalized from rings to algebraic theories by \cite{dukarm}. See also \cite{adamek}. More recently, topos theorists have defined theories to be Morita equivalent just in case their classifying toposes are equivalent \citep{johnstone}. See \cite{tsementzis} for a comparison of the topos-theoretic notion of Morita equivalence with ours.} This criterion is a natural generalization of definitional equivalence. In fact, Morita equivalence is essentially the same as definitional equivalence, except that it allows one to define new sort symbols in addition to new predicate symbols, function symbols, and constant symbols. In order to state the criterion precisely, we again need to do some work. We begin by defining the concept of a Morita extension. We then make precise the sense in which Morita equivalence is a natural generalization of definitional equivalence by proving analogues of Theorems 3.1, 3.2 and 3.3.

\subsection{Morita extensions}

As we did for predicates, functions, and constants, we need to say how to define new sorts. Let $\Sigma\subset\Sigma^+$ be signatures and consider a sort symbol $\sigma\in\Sigma^+-\Sigma$. One can define the sort $\sigma$ as a product sort, a coproduct sort, a subsort, or a quotient sort. In each case, one defines $\sigma$ using old sorts in $\Sigma$ and new function symbols in $\Sigma^+-\Sigma$. These new function symbols specify how the new sort $\sigma$ is related to the old sorts in $\Sigma$. We describe these four cases in detail.

In order to define $\sigma$ as a product sort, one needs two function symbols $\pi_1,\pi_2\in\Sigma^+-\Sigma$ with $\pi_1$ of arity $\sigma\rightarrow\sigma_1$, $\pi_2$ of arity $\sigma\rightarrow\sigma_2$, and $\sigma_1,\sigma_2\in\Sigma$. The function symbols $\pi_1$ and $\pi_2$ serve as the ``canonical projections'' associated with the product sort $\sigma$. An explicit definition of the symbols $\sigma, \pi_1$, and $\pi_2$ as a \textbf{product sort} in terms of $\Sigma$ is a $\Sigma^+$-sentence of the form
$$
\forall_{\sigma_1}x\forall_{\sigma_2} y\exists_{\sigma=1}z(\pi_1(z)=x\land\pi_2(z)=y)
$$
One should think of a product sort $\sigma$ as the sort whose elements are ordered pairs, where the first element of each pair is of sort $\sigma_1$ and the second is of sort $\sigma_2$.

One can also define $\sigma$ as a coproduct sort. One again needs two function symbols $\rho_1,\rho_2\in\Sigma^+-\Sigma$ with $\rho_1$ of arity $\sigma_1\rightarrow\sigma$, $\rho_2$ of arity $\sigma_2\rightarrow\sigma$, and $\sigma_1,\sigma_2\in\Sigma$. The function symbols $\rho_1$ and $\rho_2$ are the ``canonical injections'' associated with the coproduct sort $\sigma$. An explicit definition of the symbols $\sigma, \rho_1$, and $\rho_2$ as a \textbf{coproduct sort} in terms of $\Sigma$ is a $\Sigma^+$-sentence of the form
\begin{gather*}
\forall_\sigma z\big(\exists_{\sigma_1=1}x(\rho_1(x)=z)\lor\exists_{\sigma_2=1} y(\rho_2(y)=z)\big)
\land \forall_{\sigma_1} x\forall_{\sigma_2} y\lnot\big(\rho_1(x)=\rho_2(y)\big)
\end{gather*}
One should think of a coproduct sort $\sigma$ as the disjoint union of the elements of sorts $\sigma_1$ and $\sigma_2$.

When defining a new sort $\sigma$ as a product sort or a coproduct sort, one uses two sort symbols in $\Sigma$ and two function symbols in $\Sigma^+-\Sigma$. The next two ways of defining a new sort $\sigma$ only require one sort symbol in $\Sigma$ and one function symbol in $\Sigma^+-\Sigma$. 

In order to define $\sigma$ as a subsort, one needs a function symbol $i\in\Sigma^+-\Sigma$ of arity $\sigma\rightarrow\sigma_1$ with $\sigma_1\in\Sigma$. The function symbol $i$ is the ``canonical inclusion'' associated with the subsort $\sigma$. An explicit definition of the symbols $\sigma$ and $i$ as a \textbf{subsort} in terms of $\Sigma$ is a $\Sigma^+$-sentence of the form
\begin{gather}
\label{subsortdefinition}
\forall_{\sigma_1}x\big(\phi(x)\leftrightarrow\exists_\sigma z(i(z)=x)\big)
\land\forall_\sigma z_1\forall_\sigma z_2\big(i(z_1)=i(z_2)\rightarrow z_1=z_2\big)
\end{gather}
where $\phi(x)$ is a $\Sigma$-formula. One can think of the subsort $\sigma$ as consisting of ``the elements of sort $\sigma_1$ that are $\phi$.'' The sentence \eqref{subsortdefinition} entails the  $\Sigma$-sentence $\exists_{\sigma_1} x \phi(x)$. As before, we will call this $\Sigma$-sentence the \textbf{admissibility condition} for the definition \eqref{subsortdefinition}.

Lastly, in order to define $\sigma$ as a quotient sort one needs a function symbol $\epsilon\in\Sigma^+-\Sigma$ of arity $\sigma_1\rightarrow\sigma$ with $\sigma_1\in\Sigma$. An explicit definition of the symbols $\sigma$ and $\epsilon$ as a \textbf{quotient sort} in terms of $\Sigma$ is a $\Sigma^+$-sentence of the form
\begin{equation}
\label{quotientdefinition}
\forall_{\sigma_1} x_1\forall_{\sigma_1}x_2\big(\epsilon(x_1)=\epsilon(x_2)\leftrightarrow\phi(x_1,x_2)\big) \land\forall_\sigma z\exists_{\sigma_1}x(\epsilon(x)=z)
\end{equation}
where $\phi(x_1,x_2)$ is a $\Sigma$-formula. This sentence defines $\sigma$ as a quotient sort that is obtained by ``quotienting out'' the sort $\sigma_1$ with respect to the formula $\phi(x_1, x_2)$. The sort $\sigma$ should be thought of as the set of ``equivalence classes of elements of $\sigma_1$ with respect to the relation $\phi(x_1,x_2)$.'' The function symbol $\epsilon$ is the ``canonical projection'' that maps an element to its equivalence class. One can verify that the sentence \eqref{quotientdefinition} implies that $\phi(x_1,x_2)$ is an equivalence relation. In particular, it entails the following $\Sigma$-sentences:
$$
\begin{aligned}
&\forall_{\sigma_1}x(\phi(x,x))\\
&\forall_{\sigma_1}x_1\forall_{\sigma_1}x_2(\phi(x_1,x_2)\rightarrow\phi(x_2, x_1))\\
&\forall_{\sigma_1}x_1\forall_{\sigma_1}x_2\forall_{\sigma_1}x_3\big((\phi(x_1,x_2)\land\phi(x_2,x_3))\rightarrow\phi(x_1, x_3)\big)
\end{aligned}
$$
These $\Sigma$-sentences are the \textbf{admissibility conditions} for the definition \eqref{quotientdefinition}.

Now that we have presented the four ways of defining new sort symbols, we can define the concept of a Morita extension. A Morita extension is a natural generalization of a definitional extension. The only difference is that now one is allowed to define new sort symbols. Let $\Sigma\subset\Sigma^+$ be signatures and $T$ a $\Sigma$-theory. A \textbf{Morita extension} of $T$ to the signature $\Sigma^+$ is a $\Sigma^+$-theory
$$
T^+=T\cup\{\delta_s:s\in\Sigma^+-\Sigma\}
$$
that satisfies the following conditions. First, for each symbol $s\in\Sigma^+-\Sigma$ the sentence $\delta_s$ is an explicit definition of $s$ in terms of $\Sigma$. Second, if $\sigma\in\Sigma^+-\Sigma$ is a sort symbol and $f\in \Sigma^+-\Sigma$ is a function symbol that is used in the explicit definition of $\sigma$, then $\delta_f=\delta_\sigma$. (For example, if $\sigma$ is defined as a product sort with projections $\pi_1$ and $\pi_2$, then $\delta_\sigma=\delta_{\pi_1}=\delta_{\pi_2}$.)
And third, if $\alpha_s$ is an admissibility condition for a definition $\delta_s$, then $T\vDash\alpha_s$. 

Note that unlike a definitional extension of a theory, a Morita extension can have more sort symbols than the original theory.\footnote{Also note that if $T^+$ is a Morita extension of $T$ to $\Sigma^+$, then there are restrictions on the arities of predicates, functions, and constants in $\Sigma^+-\Sigma$. If $p\in\Sigma^+-\Sigma$ is a predicate symbol of arity $\sigma_1\times\ldots\times\sigma_n$, we immediately see that $\sigma_1,\ldots,\sigma_n\in\Sigma$. Taking a single Morita extension does not allow one to define predicate symbols that apply to sorts that are not in $\Sigma$. One must take multiple Morita extensions to do this.  Likewise, any constant symbol $c\in\Sigma^+-\Sigma$ must be of sort $\sigma\in\Sigma$. And a function symbol $f\in\Sigma^+-\Sigma$ must either have arity $\sigma_1\times\ldots\times\sigma_n\rightarrow\sigma$ with $\sigma_1,\ldots,\sigma_n,\sigma\in\Sigma$, or $f$ must be one of the function symbols that appears in the definition of a new sort symbol $\sigma\in\Sigma^+-\Sigma$.} The following is a particularly simple example of a Morita extension.

\begin{example}
\label{extensionexample}
Let $\Sigma=\{\sigma, p\}$ and $\Sigma^+=\{\sigma, \sigma^+, p, i\}$ be a signatures with $\sigma$ and $\sigma^+$ sort symbols, $p$ a predicate symbol of arity $\sigma$, and $i$ a function symbol of arity $\sigma^+\rightarrow\sigma$. Consider the $\Sigma$-theory $T=\{\exists_\sigma x p(x)\}$. The following $\Sigma^+$-sentence defines the sort symbol $\sigma^+$ as the subsort consisting of ``the elements that are $p$.''
\begin{align*}
&\forall_{\sigma}x\big(p(x)\leftrightarrow\exists_{\sigma^+} z(i(z)=x)\big)\land\forall_{\sigma^+}z_1\forall_{\sigma^+}z_2\big(i(z_1)=i(z_2)\rightarrow z_1=z_2\big)\tag{$\delta_{\sigma^+}$}
\end{align*}
The $\Sigma^+$-theory $T^+=T\cup\{\delta_{\sigma^+}\}$ is a Morita extension of $T$ to the signature $\Sigma^+$. The theory $T^+$ adds to the theory $T$ the ability to quantify over the set of ``things that are $p$.'' 
\end{example}

\subsection{Three results}

As with a definitional extension, a Morita extension ``says no more'' than the original theory. We will make this idea precise by proving analogues of Theorems 3.1, 3.2, and 3.3. These three results also demonstrate how closely related the concept of a Morita extension is to that of a definitional extension.

Theorem 3.1 generalizes in a perfectly natural way. When $T^+$ is a Morita extension of $T$, the models of $T^+$ are ``determined'' by the models of $T$.
\begin{theorem}
Let $\Sigma\subset\Sigma^+$ be signatures and $T$ a $\Sigma$-theory. If $T^+$ is a Morita extension of $T$ to $\Sigma^+$, then every model $M$ of $T$ has a unique expansion (up to isomorphism) $M^+$ that is a model of $T^+$.
\end{theorem}

Before proving Theorem 4.1, we introduce some notation and prove a lemma. Suppose that a $\Sigma^+$-theory $T^+$ is a Morita extension of a $\Sigma$-theory $T$. Let $M$ and $N$ be models of $T^+$ with $h:M|_\Sigma\rightarrow N|_\Sigma$ an elementary embedding between the $\Sigma$-structures $M|_\Sigma$ and $N|_\Sigma$. The elementary embedding $h$ naturally induces a map $h^+:M\rightarrow N$
between the $\Sigma^+$-structures $M$ and $N$.

We know that $h$ is a family of maps $h_\sigma: M_\sigma\rightarrow N_\sigma$ for each sort $\sigma\in\Sigma$. In order to describe $h^+$ we need to describe the map $h^+_\sigma:M_\sigma\rightarrow N_\sigma$ for each sort $\sigma\in\Sigma^+$. If $\sigma\in\Sigma$, we simply let $h_\sigma^+=h_\sigma$.  On the other hand, when $\sigma\in\Sigma^+-\Sigma$, there are four cases to consider.  We describe $h^+_\sigma$ in the cases where the theory $T^+$ defines $\sigma$ as a product sort or a subsort. The coproduct and quotient sort cases are described analogously.

First, suppose that $T^+$ defines $\sigma$ as a product sort. Let $\pi_1,\pi_2\in\Sigma^+$ be the projections of arity $\sigma\rightarrow\sigma_1$ and $\sigma\rightarrow\sigma_2$ with $\sigma_1,\sigma_2\in\Sigma$. The definition of the function $h^+_\sigma$ is suggested by the following diagram.
$$
\begin{tikzpicture}
\node(1) {$M_\sigma$};
\node(2) [below right of=1] {$M_{\sigma_1}$};
\node(3) [right of =2] {$N_{\sigma_1}$};
\node(4) [above right of=3] {$N_{\sigma}$};
\node(5) [above right of=1] {$M_{\sigma_2}$};
\node(6) [right of=5] {$N_{\sigma_2}$};
\draw[->, dotted] (1) to node  {$h^+_\sigma$} (4);
\draw[->] (1) to node [swap] {$\pi_1^M$} (2);
\draw[->] (2) to node {$h^+_{\sigma_1}$} (3);
\draw[->] (4) to node {$\pi_1^N$} (3);
\draw[->] (1) to node {$\pi_2^M$} (5);
\draw[->] (5) to node {$h^+_{\sigma_2}$} (6);
\draw[->] (4) to node [swap] {$\pi_2^N$} (6);
\end{tikzpicture}
$$
Let $m\in M_\sigma$. We define $h^+_\sigma(m)$ to be the unique $n\in N_\sigma$ that satisfies both $\pi_1^N(n)=h^+_{\sigma_1}\circ\pi_1^M(m)$ and $\pi_2^N(n)=h^+_{\sigma_2}\circ\pi_2^M(m)$. We know that such an $n$ exists and is unique because $N$ is a model of $T^+$ and $T^+$ defines the symbols $\sigma$, $\pi_1$, and $\pi_2$ to be a product sort. One can verify that this definition of $h^+_\sigma$ makes the above diagram commute. 

Suppose, on the other hand, that $T^+$ defines $\sigma$ as the subsort of ``elements of sort $\sigma_1$ that are $\phi$.'' Let $i\in\Sigma^+$ be the inclusion map of arity $\sigma\rightarrow\sigma_1$ with $\sigma_1\in\Sigma$. As above, the definition of $h^+_\sigma$ is suggested by the following diagram.
$$
\begin{tikzpicture}
\node(1) {$M_\sigma$};
\node(2) [below right of=1] {$M_{\sigma_1}$};
\node(3) [right of =2] {$N_{\sigma_1}$};
\node(4) [above right of=3] {$N_{\sigma}$};
\draw[->, dotted] (1) to node  {$h^+_\sigma$} (4);
\draw[->] (1) to node [swap] {$i^M$} (2);
\draw[->] (2) to node {$h^+_{\sigma_1}$} (3);
\draw[->] (4) to node {$i^N$} (3);
\end{tikzpicture}
$$
Let $m\in M_\sigma$. We see that following implications hold:
\begin{align*}
M\vDash\phi[i^M(m)]&\Rightarrow M|_\Sigma\vDash\phi[i^M(m)]\\
&\Rightarrow N|_\Sigma\vDash\phi[h^+_{\sigma_1}( i^M(m))]\Rightarrow N\vDash\phi[h^+_{\sigma_1}( i^M(m))]
\end{align*}
The first and third implications hold since $\phi(x)$ is a $\Sigma$-formula, and the second holds because $h_{\sigma_1}=h^+_{\sigma_1}$ and $h$ is an elementary embedding. $T^+$ defines the symbols $i$ and $\sigma$ as a subsort and $M$ is a model of $T^+$, so it must be that $M\vDash \phi[i^M(m)]$. By the above implications, we see that $N\vDash\phi[h^+_{\sigma_1}( i^M(m))]$. Since $N$ is also a model of $T^+$, there is a unique $n\in N_\sigma$ that satisfies $i^N(n)=h^+_{\sigma_1}(i^M(m))$. We define $h^+_\sigma(m)=n$. This definition of $h^+_\sigma$ again makes the above diagram commute. 

When $T^+$ defines $\sigma$ as a coproduct sort or a quotient sort one describes the map $h^+_\sigma$ analogously. 
For the purposes of proving Theorem 4.1, we need the following simple lemma about this map $h^+$.

\begin{lemma}
If $h:M|_\Sigma\rightarrow N|_\Sigma$ is an isomorphism, then $h^+:M\rightarrow N$ is an isomorphism.
\end{lemma}

\begin{proof}
We know that $h_\sigma:M_\sigma\rightarrow N_\sigma$ is a bijection for each $\sigma\in\Sigma$. Using this fact and the definition of $h^+$, one can verify that $h^+_\sigma:M_\sigma\rightarrow N_\sigma$ is a bijection for each sort $\sigma\in\Sigma^+$. So $h^+$ is a family of bijections. And furthermore, the commutativity of the above diagrams implies that $h^+$ preserves any function symbols 
that are used to define new sorts.

It only remains to check that $h^+$ preserves predicates, functions, and constants that have arities and sorts in $\Sigma$. Since $h:M|_\Sigma\rightarrow N|_\Sigma$ is a isomorphism, we know that $h^+$ preserves the symbols in $\Sigma$. So let $p\in\Sigma^+-\Sigma$ be a predicate symbol of arity $\sigma_1\times\ldots\times\sigma_n$ with $\sigma_1,\ldots,\sigma_n\in\Sigma$. There must be a $\Sigma$-formula $\phi(x_1,\ldots,x_n)$ such that $T^+\vDash\forall_{\sigma_1}x_1\ldots\forall_{\sigma_n}x_n(p(x_1,\ldots,x_n)\leftrightarrow\phi(x_1,\ldots,x_n))$. We know that $h:M|_\Sigma\rightarrow N|_\Sigma$ is an elementary embedding, so in particular it preserves the formula $\phi(x_1,\ldots, x_n)$. This implies that $(m_1,\ldots,m_n)\in p^M$ if and only if $(h_{\sigma_1}(m_1),\ldots,h_{\sigma_n}(m_n))\in p^N$. Since $h^+_{\sigma_i}=h_{\sigma_i}$ for each $i=1,\ldots, n$, it must be that $h^+$ also preserves the predicate $p$. An analogous argument demonstrates that $h^+$ preserves functions and constants.
\end{proof}

We now turn to the proof of Theorem 4.1. 
\begin{proof}[Proof of Theorem 4.1]
Let $M$ be a model of $T$. First note that if $M^+$ exists, then it is unique up to isomorphism. For if $N$ is a model of $T^+$ with $N|_\Sigma=M$, then by letting $h$ be the identity map (which is an isomorphism) Lemma 4.1 implies that $M^+\cong N$. We need only define the $\Sigma^+$-structure $M^+$. To guarantee that $M^+$ is an expansion of $M$ we interpret every symbol in $\Sigma$ the same way that $M$ does. We need to say how the symbols in $\Sigma^+-\Sigma$ are interpreted. There are a number of cases to consider.

Suppose that $p\in\Sigma^+-\Sigma$ is a predicate symbol of arity $\sigma_1\times\ldots\times\sigma_n$ with $\sigma_1,\ldots,\sigma_n\in\Sigma$. There must be a $\Sigma$-formula $\phi(x_1,\ldots, x_n)$ such that $T^+\vDash\forall_{\sigma_1}x_1\ldots\forall_{\sigma_n}x_n(p(x_1,\ldots, x_n)\leftrightarrow\phi(x_1,\ldots, x_n))$. We define the interpretation of the symbol $p$ in $M^+$ by letting $(a_1,\ldots, a_n)\in p^{M^+}$ if and only if $M\vDash\phi[a_1,\ldots, a_n]$. It is easy to see that this definition of $p^A$ implies that $M^+\vDash\delta_p$. The cases of function and constant symbols are handled similarly.

Let $\sigma\in\Sigma^+-\Sigma$ be a sort symbol. We describe the cases where $T^+$ defines $\sigma$ as a product sort or a subsort. The coproduct and quotient sort cases follow analogously. Suppose first that $\sigma$ is defined as a product sort with $\pi_1$ and $\pi_2$ the projections of arity $\sigma\rightarrow\sigma_1$ and $\sigma\rightarrow\sigma_2$, respectively. We define $M_\sigma^+=M^+_{\sigma_1}\times M^+_{\sigma_2}$ with $\pi_1^{M^+}:M^+_\sigma\rightarrow M^+_{\sigma_1}$ and $\pi_2^{M^+}:M^+_\sigma\rightarrow M^+_{\sigma_2}$ the canonical projections. One can easily verify that $M^+\vDash\delta_\sigma$. On the other hand, suppose that $\sigma$ is defined as a subsort with defining $\Sigma$-formula $\phi(x)$ and inclusion $i$ of arity $\sigma\rightarrow\sigma_1$. We define $M_\sigma^+=\{a\in M_{\sigma_1}: M\vDash\phi[a]\}$ with $i^{M^+}:M_\sigma^+\rightarrow M^+_{\sigma_1}$ the inclusion map. One can again verify that $M^+\vDash\delta_\sigma$. 
\end{proof}

We have shown that the exact analogue of Theorem 3.1 holds for Morita extensions. Theorem 3.2 also generalizes in a perfectly natural way. Indeed, the generalization follows as a simple corollary to Theorem 4.1.

\begin{theorem}
If $T^+$ is a Morita extension of $T$, then $T^+$ is a conservative extension of $T$.
\end{theorem}

\begin{proof}
Suppose that $T^+$ is not a conservative extension of $T$. One can easily see that $T\vDash\phi$ implies that $T^+\vDash\phi$ for every $\Sigma$-sentence $\phi$. So there must be some $\Sigma$-sentence $\phi$ such that $T^+\vDash\phi$, but $T\not\vDash\phi$. This implies that there is a model $M$ of $T$ such that $M\vDash\lnot\phi$. This model $M$ has no expansion that is a model of $T^+$ since $T^+\vDash\phi$, contradicting Theorem 4.1.
\end{proof}

Theorems 3.1 and 3.2 therefore generalize naturally from definitional extensions to Morita extensions. In order to generalize Theorem 3.3, however, we need to do some work. Theorem 3.3 said that if $T^+$ is a definitional extension of $T$ to $\Sigma^+$, then for every $\Sigma^+$-formula $\phi(x_1,\ldots, x_n)$ there is a corresponding formula $\phi^*(x_1,\ldots, x_n)$ that is equivalent to $\phi(x_1,\ldots, x_n)$ according to the theory $T^+$. The following example demonstrates that this result does not generalize to the case of Morita extensions in a perfectly straightforward manner. 

\begin{example}
Recall the theories $T$ and $T^+$ from Example \ref{extensionexample} and consider the $\Sigma^+$-formula $\phi(x, z)$ defined by $i(z)=x$. One can easily see that there is no $\Sigma$-formula $\phi^*(x, z)$ that is equivalent to $\phi(x, z)$ according to the theory $T^+$. Indeed, the variable $z$ cannot appear in any $\Sigma$-formula since it is of sort $\sigma^+\in\Sigma^+-\Sigma$. A $\Sigma$-formula simply cannot say how variables with sorts in $\Sigma$ relate to variables with sorts in $\Sigma^+$.
\end{example}

In order to generalize Theorem 3.3, therefore, we need a way of specifying how variables with sorts in $\Sigma^+-\Sigma$ relate to variables with sorts in $\Sigma$. We do this by defining the concept of a ``code.''\footnote{One can compare this concept with the one employed by \cite{szczerba1977}.} Let $\Sigma\subset\Sigma^+$ be signatures with $T$ a $\Sigma$-theory and $T^+$ a Morita extension of $T$ to $\Sigma^+$. A \textbf{code} for the variables $x_1,\ldots, x_n$ of sorts $\sigma_1,\ldots, \sigma_n\in\Sigma^+-\Sigma$ is a $\Sigma^+$-formula
$$
\xi_1(x_1, y_{11}, y_{12})\land\ldots\land\xi_n(x_n, y_{n1}, y_{n2})
$$
where the conjuncts $\xi_i$ are defined as follows. Suppose that $T^+$ defines $\sigma_i$ as a product sort with $\pi_1$ and $\pi_2$ the projections of arity $\sigma_i\rightarrow\sigma_{i1}$ and $\sigma_i\rightarrow\sigma_{i2}$. The conjunct $\xi_i(x_i, y_{i1}, y_{i2})$ is then the $\Sigma^+$-formula $\pi_1(x_i)=y_{i1}\land\pi_2(x_i)=y_{i2},$ where $y_{i1}$ and $y_{i2}$ are variables of sorts $\sigma_{i1},\sigma_{i2}\in\Sigma$. On the other hand, suppose that $T^+$ defines $\sigma_i$ as a coproduct sort with injections $\rho_1$ and $\rho_2$ of arity $\sigma_{i1}\rightarrow\sigma_i$ and $\sigma_{i2}\rightarrow\sigma_i$. Then the conjunct $\xi_i$ is either the $\Sigma^+$-formula $\rho_1(y_{i1})=x_i$ or the $\Sigma^+$-formula $\rho_2(y_{i2})=x_i$, where $y_{i1}$ and $y_{i2}$ are again variables of sorts $\sigma_{i1},\sigma_{i2}\in\Sigma$.

The subsort and quotient sort cases are handled analogously. Suppose that $T^+$ defines $\sigma_i$ as a subsort with $i$ the inclusion map of arity $\sigma_i\rightarrow\sigma_{i1}$. Then the conjunct $\xi_i$ is the $\Sigma^+$-formula $i(x_i)=y_{i1}$, where $y_{i1}$ is a variable of sort $\sigma_{i1}\in\Sigma$. And finally, suppose that $T^+$ defines $\sigma_i$ as a quotient sort with $\epsilon$ the projection of arity $\sigma_{i1}\rightarrow\sigma_i$. The conjunct $\xi_i$ is then the $\Sigma^+$-formula $\epsilon(y_{i1})=x_i$, where $y_{i1}$ is again a variable of sort $\sigma_{i1}\in\Sigma$. Given the empty sequence of variables, we let the \textbf{empty code} be the tautology $\exists_\sigma x(x=x)$, where $\sigma\in\Sigma$ is a sort symbol. 

We will use the notation $\xi(x_1, \ldots, y_{n2})$ to denote the code $\xi_1(x_1,y_{11}, y_{12})\land\ldots\land\xi_n(x_n, y_{n1}, y_{n2})$ for the variables $x_1,\ldots, x_n$. Note that the variables $y_{i1}$ and $y_{i2}$ have sorts in $\Sigma$ for each $i=1,\ldots, n$. One should think of a code $\xi(x_1,\ldots, y_{n2})$ for $x_1,\ldots, x_n$ as encoding one way that the variables $x_1,\ldots, x_n$ with sorts in $\Sigma^+-\Sigma$ might be related to variables $y_{11}, \ldots, y_{n2}$ that have sorts in $\Sigma$. One additional piece of notation will be useful in what follows. Given a $\Sigma^+$-formula $\phi$, we will write $\phi(x_1,\ldots, x_n, \overline{x}_1, \ldots, \overline{x}_m)$ to indicate that the variables $x_1,\ldots, x_n$ have sorts $\sigma_1,\ldots, \sigma_n\in\Sigma^+-\Sigma$ and that the variables $\overline{x}_1,\ldots, \overline{x}_m$ have sorts $\overline{\sigma_1},\ldots, \overline{\sigma}_m\in\Sigma$. 

We can now state our generalization of Theorem 3.3. One proves this result by induction on the complexity of $\phi(x_1,\ldots, x_n)$. The proof has been placed in an appendix. 

\begin{theorem}
Let $\Sigma\subset\Sigma^+$ be signatures and $T$ a $\Sigma$-theory. Suppose that $T^+$ is a Morita extension of $T$ to $\Sigma^+$ and that $\phi(x_1,\ldots, x_n, \overline{x}_1,\ldots, \overline{x}_m)$ is a $\Sigma^+$-formula. Then for every code $\xi(x_1,\ldots, y_{n2})$ for the variables $x_1,\ldots, x_n$ there is a $\Sigma$-formula $\phi^*(\overline{x}_1,\ldots, \overline{x}_m, y_{11}, \ldots, y_{n2})$ such that 
\begin{align*}
T^+\vDash \forall_{\sigma_1} x_1\ldots&\forall_{\sigma_n} x_n\forall_{\overline{\sigma}_1} \overline{x}_1\ldots\forall_{\overline{\sigma}_m}\overline{x}_m\forall_{\sigma_{11}} y_{11}\ldots\forall_{\sigma_{n2}} y_{n2}\big(\xi(x_1,\ldots, y_{n2})\rightarrow\\
&(\phi(x_1,\ldots x_n, \overline{x}_1,\ldots, \overline{x}_m)\leftrightarrow\phi^*(\overline{x}_1,\ldots,\overline{x}_m, y_{11},\ldots, y_{n2}))\big)
\end{align*}
\end{theorem}

The idea behind Theorem 4.3 is simple. Although one might not initially be able to translate a $\Sigma^+$-formula $\phi$ into an equivalent $\Sigma$-formula $\phi^*$, such a translation is possible after one specifies how the variables in $\phi$ with sorts in $\Sigma^+-\Sigma$ are related to variables with sorts in $\Sigma$. Theorem 4.3 has the following immediate corollary.

\begin{corollary}
Let $\Sigma\subset\Sigma^+$ be signatures and $T$ a $\Sigma$-theory. If $T^+$ is a Morita extension of $T$ to $\Sigma^+$, then for every $\Sigma^+$-sentence $\phi$ there is a $\Sigma$-sentence $\phi^*$ such that $T^+\vDash\phi\leftrightarrow\phi^*$.
\end{corollary}

\begin{proof}
Let $\phi$ be a $\Sigma^+$-sentence and consider the empty code $\xi$. Theorem 4.3 implies that there is a $\Sigma$-sentence $\phi^*$ such that $T^+\vDash \xi\rightarrow (\phi\leftrightarrow\phi^*)$. Since $\xi$ is a tautology we trivially have that $T^+\vDash \phi\leftrightarrow\phi^*$.
\end{proof}

Theorems 4.1, 4.2, and 4.3 capture different senses in which a Morita extension of a theory ``says no more'' than the original theory. The definition of Morita equivalence is exactly analogous to definitional equivalence.

\begin{definition}
Let $T_1$ be a $\Sigma_1$-theory and $T_2$ a $\Sigma_2$-theory. $T_1$ and $T_2$ are \textbf{Morita equivalent} if there are theories $T_1^1, \ldots, T_1^n$ and $T_2^1,\ldots, T_2^m$ that satisfy the following three conditions:
\begin{itemize}
\item Each theory $T_1^{i+1}$ is a Morita extension of $T_1^{i}$,
\item Each theory $T_2^{i+1}$ is a Morita extension of $T_2^i$,
\item $T_1^n$ and $T_2^m$ are logically equivalent $\Sigma$-theories with $\Sigma_1\cup\Sigma_2\subset\Sigma$.
\end{itemize}
\end{definition}

Two theories are Morita equivalent if they have a ``common Morita extension.'' The situation can be pictured as follows, where each arrow in the figure indicates a Morita extension.
\begin{center}
\includegraphics{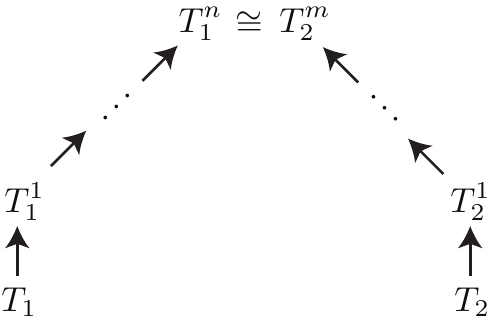}
\end{center}

At first glance, Morita equivalence might strike one as different from definitional equivalence in an important way. To show that theories are Morita equivalent, one is allowed to take any finite number of Morita extensions of the theories. On the other hand, to show that two theories are definitionally equivalent, it appears that one is only allowed to take \textit{one} definitional extension of each theory. One might worry that Morita equivalence is therefore not perfectly analogous to definitional equivalence. 

Fortunately, this is not the case. Theorem 3.3 implies that if theories $T_1, \ldots, T_n$ are such that each $T_{i+1}$ is a definitional extension of $T_i$, then $T_n$ is in fact a definitional extension of $T_1$. (One can easily verify that this is not true of Morita extensions.) To show that two theories are definitionally equivalent, therefore, one actually \textit{is} allowed to take any finite number of definitional extensions of each theory.

If two theories are definitionally equivalent, then they are trivially Morita equivalent. Unlike definitional equivalence, however, Morita equivalence is capable of capturing a sense in which theories with different sort symbols are equivalent. The following example demonstrates that Morita equivalence is a more liberal criterion for theoretical equivalence.

\begin{example}
Recall the $\Sigma_1$-theory $T_1$ and the $\Sigma_2$-theory $T_2$ from Example \ref{motivating example}. These theories are not definitionally equivalent, but they are Morita equivalent. Let $\Sigma=\Sigma_1\cup\Sigma_2\cup\{i_2, i_3\}$ be a signature with $i_2$ and $i_3$ function symbols of arity $\sigma_2\rightarrow\sigma_1$ and $\sigma_3\rightarrow\sigma_1$. Consider the following $\Sigma$-sentences.
\begin{align*}
&\begin{aligned}
&\forall_{\sigma_1}x\big(p(x)\leftrightarrow\exists_{\sigma_2} y(i_2(y)=x)\big)\\
&\qquad\land\forall_{\sigma_2} y_1\forall_{\sigma_2} y_2\big(i_2(y_1)=i_2(y_2)\rightarrow y_1=y_2\big)\\
\end{aligned}
\tag{$\delta_{\sigma_2}$}\\
&\begin{aligned}
&\forall_{\sigma_1}x\big(q(x)\leftrightarrow\exists_{\sigma_3} z(i_3(z)=x)\big)\\
&\qquad\land\forall_{\sigma_3} z_1\forall_{\sigma_3} z_2\big(i_3(z_1)=i_3(z_2)\rightarrow z_1=z_2\big)\\
\end{aligned}
\tag{$\delta_{\sigma_3}$}\\
%
%
%
&\begin{aligned}
&\forall_{\sigma_1} x\big(\exists_{\sigma_2=1}y(i_2(y)=x)\lor\exists_{\sigma_3=1} z(i_3(z)=x)\big)\\
&\qquad\land \forall_{\sigma_2} y\forall_{\sigma_3} z \lnot\big(i_2(y)=i_3(z)\big)\\
\end{aligned}
\tag{$\delta_{\sigma_1}$}\\
&\begin{aligned}
&\forall_{\sigma_1} x\big(p(x)\leftrightarrow\exists_{\sigma_2} y(i_2(y)=x)\big)
\end{aligned} 
\tag{$\delta_p$}\\
&\begin{aligned}
&\forall_{\sigma_1} x\big(q(x)\leftrightarrow\exists_{\sigma_3} z(i_3(z)=x)\big)
\end{aligned}
\tag{$\delta_q$}
\end{align*}
The $\Sigma$-theory $T_1^1=T_1\cup\{\delta_{\sigma_2}, \delta_{\sigma_3}\}$ is a Morita extension of $T_1$ to the signature $\Sigma$. It defines $\sigma_2$ and $i_2$ to be the subsort of ``elements that are $p$'' and $\sigma_3$ and $i_3$ to be the subsort of ``elements that are $q$.''
The theory $T_2^1=T_2\cup\{\delta_{\sigma_1}\}$ is a Morita extension of $T_2$ to the signature $\Sigma_2\cup\{\sigma_1, i_2, i_3\}$. It defines $\sigma_1$ to be the coproduct sort of $\sigma_2$ and $\sigma_3$. Lastly, the $\Sigma$-theory $T_2^2=T_2^1\cup\{\delta_p, \delta_q\}$ is a Morita extension of $T_2^1$ to the signature $\Sigma$. It defines the predicates $p$ and $q$ to apply to elements in the ``images'' of $i_2$ and $i_3$, respectively.
%
One can verify that $T_1^1$ and $T_2^2$ are logically equivalent, so $T_1$ and $T_2$ are Morita equivalent.
\end{example}

\section{Categorical Equivalence}

Morita equivalence captures a clear and robust sense in which theories might be equivalent, but it is a difficult criterion to apply outside of the framework of first-order logic. Indeed, without a formal language one does not have the resources to say what an explicit definition is. Questions of equivalence and inequivalence of theories, however, still come up outside of this framework. It is well known, for example, that there are different ways of formulating the theory of smooth manifolds \citep{nestruev2002}. There are also different formulations of the theory of topological spaces \citep{kuratowski1966}. None of these formulations are first-order theories. Physical theories too are rarely formulated in first-order logic, and there are many pairs of physical theories that are often considered equivalent.\footnote{For example, see \cite{glymour1977}, \cite{knox2013}, and \cite{weatherallunpublished} for discussion of whether or not Newtonian gravitation and geometrized Newtonian gravitation are equivalent. See \cite{north2009}, \cite{halvorsonunpublished}, \cite{swansonhalvorsonunpublished}, \cite{curiel2014}, and \cite{barrett2014} for discussion of whether or not Hamiltonian and Lagrangian mechanics are equivalent. See \cite{rosenstockbarrettweatherall} for a discussion of general relativity and the theory of Einstein algebras and \cite{weatherall2015} for a summary of many of these results.}

Morita equivalence is incapable of capturing any sense in which these theories are equivalent. We need a criterion for theoretical equivalence that is applicable outside the framework of first-order logic. \textit{Categorical equivalence} is one such criterion.\footnote{The reader is encouraged to consult \cite{cwm}, \cite{borceux1994}, or \cite{awodey2010} for preliminaries.} It was first described by \cite{eilenbergmaclane1942, eilenbergmaclane1945}, but was only recently introduced into philosophy of science by \cite{halvorson2012, halvorson2015} and \cite{weatherallunpublished}. In this section, we describe categorical equivalence and then show how it is related to Morita equivalence.

Categorical equivalence is motivated by the following simple observation: First-order theories have categories of models. A \textbf{category} $C$ is a collection of objects with arrows between the objects that satisfy two basic properties. First, there is an associative composition operation $\circ$ defined on the arrows of $C$, and second, every object $c$ in $C$ has an identity arrow $1_c:c\rightarrow c$.  If $T$ is a $\Sigma$-theory, we will use the notation $\text{Mod}(T)$ to denote the \textbf{category of models} of $T$. An object in $\text{Mod}(T)$ is a model $M$ of $T$, and an arrow $f:M\rightarrow N$ between objects in $\text{Mod}(T)$ is an elementary embedding $f:M\rightarrow N$ between the models $M$ and $N$. One can easily verify that $\text{Mod}(T)$ is a category.

Before describing categorical equivalence, we need some additional terminology. Let $C$ and $D$ be categories. A \textbf{functor} $F:C\rightarrow D$ is a map from objects and arrows of $C$ to objects and arrows of $D$ that satisfies
$$
F(f:a\rightarrow b)=Ff:Fa\rightarrow Fb\qquad F(1_c)=1_{Fc}\qquad F(g\circ h)=Fg\circ Fh
$$ 
for every arrow $f:a\rightarrow b$ in $C$, every object $c$ in $C$, and every composable pair of arrows $g$ and $h$ in $C$. Functors are the ``structure-preserving maps'' between categories; they preserve domains, codomains, identity arrows, and the composition operation. A functor $F:C\rightarrow D$ is \textbf{full} if for all objects $c_1, c_2$ in $C$ and arrows $g:Fc_1\rightarrow Fc_2$ in $D$ there exists an arrow $f:c_1\rightarrow c_2$ in $C$ with $Ff=g$. $F$ is \textbf{faithful} if $Ff=Fg$ implies that $f=g$ for all arrows $f:c_1\rightarrow c_2$ and $g:c_1\rightarrow c_2$ in $C$. $F$ is \textbf{essentially surjective} if for every object $d$ in $D$ there exists an object $c$ in $C$ such that $Fc\cong d$. A functor $F:C\rightarrow D$ that is full, faithful, and essentially surjective is called an \textbf{equivalence of categories}. The categories $C$ and $D$ are \textbf{equivalent} if there exists an equivalence between them.\footnote{The concept of a ``natural transformation'' is often used to define when two categories are equivalent. $C$ and $D$ are equivalent if there are functors $F:C\rightarrow D$ and $G:D\rightarrow C$ such that $FG$ is naturally isomorphic to the identity functor $1_D$ and $GF$ is naturally isomorphic to $1_C$. See \cite{cwm} for the definition of a natural transformation and for proof that these two characterizations of equivalence are the same.}

A first-order theory $T$ has a category of models $\text{Mod}(T)$. This categorical structure, however, is not particular to first-order theories. Indeed, one can easily define categories of models for the different formulations of the theory of smooth manifolds and for the different formulations of the theory of topological spaces. The arrows in these categories are simply the structure-preserving maps between the objects in the categories. One can also define categories of models for physical theories.\footnote{See the examples in \citep{weatherallunpublished, weatherall2015, weatherallgauge} and \cite{rosenstockbarrettweatherall}.} This means that the following criterion for theoretical equivalence is applicable in a more general setting than definitional equivalence and Morita equivalence. In particular, it can be applied outside of the framework of first-order logic.

\begin{definition}
Theories $T_1$ and $T_2$ are \textbf{categorically equivalent} if their categories of models $\text{Mod}(T_1)$ and $\text{Mod}(T_2)$ are equivalent.
\end{definition}

Categorical equivalence captures a sense in which theories have ``isomorphic semantic structure.'' If $T_1$ and $T_2$ are categorically equivalent, then the relationships that models of $T_1$ bear to one another are ``isomorphic'' to the relationships that models of $T_2$ bear to one another.

In order to show how categorical equivalence relates to Morita equivalence, we focus on first-order theories. We will show that categorical equivalence is a strictly weaker criterion for theoretical equivalence than Morita equivalence is. We first need some preliminaries about the category of models $\text{Mod}(T)$ for a first-order theory $T$. Suppose that $\Sigma\subset\Sigma^+$ are signatures and that the $\Sigma^+$-theory $T^+$ is an extension of the $\Sigma$-theory $T$. There is a natural ``projection'' functor $\Pi:\text{Mod}(T^+)\rightarrow \text{Mod}(T)$ from the category of models of $T^+$ to the category of models of $T$. The functor $\Pi$ is defined as follows.
\begin{itemize}
\item $\Pi(M)=M|_\Sigma$ for every object $M$ in $\text{Mod}(T^+)$.
\item $\Pi(h)=h|_\Sigma$ for every arrow $h:M\rightarrow N$ in $\text{Mod}(T^+)$, where the family of maps $h|_\Sigma$ is defined to be $h|_\Sigma=\{h_\sigma: M_\sigma\rightarrow N_\sigma\text{ such that } \sigma\in \Sigma\}$.
\end{itemize}
Since $T^+$ is an extension of $T$, the $\Sigma$-structure $\Pi(M)$ is guaranteed to be a model of $T$. Likewise, the map $\Pi(h):M|_\Sigma\rightarrow N|_\Sigma$ is guaranteed to be an elementary embedding. One can easily verify that $\Pi:\text{Mod}(T^+)\rightarrow\text{Mod}(T)$ is a functor.

The following three propositions will together establish the relationship between $\text{Mod}(T^+)$ and $\text{Mod}(T)$ when $T^+$ is a Morita extension of $T$. They imply that when $T^+$ is a Morita extension of $T$, the functor $\Pi:\text{Mod}(T^+)\rightarrow\text{Mod}(T)$ is full, faithful, and essentially surjective. The categories $\text{Mod}(T^+)$ and $\text{Mod}(T)$ are therefore equivalent.

\begin{prop}
Let $\Sigma\subset\Sigma^+$ be signatures and $T$ a $\Sigma$-theory. If $T^+$ is a Morita extension of $T$ to $\Sigma^+$, then $\Pi$ is essentially surjective.
\end{prop}

\begin{proof}
If $M$ is a model of $T$, then Theorem 4.1 implies that there is a model $M^+$ of $T^+$ that is an expansion of $M$. Since $\Pi(M^+)=M^+|_\Sigma=M$ the functor $\Pi$ is essentially surjective.
\end{proof}

\begin{prop}
Let $\Sigma\subset\Sigma^+$ be signatures and $T$ a $\Sigma$-theory. If $T^+$ is a Morita extension of $T$ to $\Sigma^+$, then $\Pi$ is faithful.
\end{prop}

\begin{proof}
Let $h:M\rightarrow N$ and $g:M\rightarrow N$ be arrows in $\text{Mod}(T^+)$ and suppose that $\Pi(h)=\Pi(g)$. We show that $h=g$. By assumption $h_\sigma=g_\sigma$ for every sort symbol $\sigma\in\Sigma$. We show that $h_\sigma=g_\sigma$ also for $\sigma\in\Sigma^+-\Sigma$. We consider the cases where $T^+$ defines $\sigma$ as a product sort or a subsort. The coproduct and quotient sort cases follow analogously. 

Suppose that $T^+$ defines $\sigma$ as a product sort with projections $\pi_1$ and $\pi_2$ of arity $\sigma\rightarrow\sigma_1$ and $\sigma\rightarrow\sigma_2$. Then the following equalities hold.
$$
\pi_1^N\circ h_\sigma=h_{\sigma_1}\circ\pi_1^M=g_{\sigma_1}\circ\pi_1^M=\pi_1^N\circ g_\sigma
$$
The first and third equalities hold since $h$ and $g$ are elementary embeddings and the second since $h_{\sigma_1}=g_{\sigma_1}$. One can verify in the same manner that $\pi_2^N\circ h_\sigma=\pi_2^N\circ g_\sigma$. Since $N$ is a model of $T^+$ and $T^+$ defines $\sigma$ as a product sort, we know that $N\vDash\forall_{\sigma_1}x\forall_{\sigma_2}y\exists_{\sigma=1} z(\pi_1(z)=x\land\pi_2(z)=y)$. This implies that $h_\sigma=g_\sigma$. 

On the other hand, if $T^+$ defines $\sigma$ as a subsort with injection $i$ of arity $\sigma\rightarrow\sigma_1$, then the following equalities hold.
$$
i^N\circ h_\sigma=h_{\sigma_1}\circ i^M=g_{\sigma_1}\circ i^M=i^N\circ g_\sigma
$$
These equalities follow in the same manner as above. Since $i^N$ is an injection it must be that $h_\sigma=g_\sigma$.
\end{proof}

Before proving that $\Pi$ is full, we need the following simple lemma. 

\begin{lemma}
Let $M$ be a model of $T^+$ with $a_1,\ldots, a_n$ elements of $M$ of sorts $\sigma_1,\ldots, \sigma_n\in\Sigma^+-\Sigma$. If $x_1,\ldots, x_n$ are variables sorts $\sigma_1,\ldots, \sigma_n$, then there is a code $\xi(x_1,\ldots, x_n, y_{11},\ldots, y_{n2})$ and elements $b_{11},\ldots, b_{n2}$ of $M$ such that $M\vDash\xi[a_1,\ldots, a_n, b_{11},\ldots, b_{n2}]$.
\end{lemma}

\begin{proof}
We define the code $\xi(x_1,\ldots, y_{n2})$. If $T^+$ defines $\sigma_i$ as a product sort, quotient sort, or subsort then we have no choice about what the conjunct $\xi_i(x_i, y_{i1}, y_{i2})$ is. If $T^+$ defines $\sigma_i$ as a coproduct sort, then we know that either there is an element $b_{i1}$ of $M$ such that $\rho_1(b_{i1})=a_i$ or there is an element $b_{i2}$ of $M$ such that $\rho_2(b_{i2})=a_i$. If the former, we let $\xi_i$ be $\rho_1(y_{i1})=x_i$ and if the latter, we let $\xi_i$ be $\rho_2(y_{i2})=x_i$. One defines the elements $b_{11}, \ldots, b_{n2}$ in the obvious way. For example, if $\sigma_i$ is a product sort, then we let $b_{i1}=\pi_1^M(a_i)$ and $b_{i2}=\pi_2^M(a_i)$. By construction, we have that $M\vDash\xi[a_1,\ldots, a_n, b_{11},\ldots, b_{n2}]$.
\end{proof}

%

We now use this lemma to show that $\Pi$ is full.

\begin{prop}
Let $\Sigma\subset\Sigma^+$ be signatures and $T$ a $\Sigma$-theory. If $T^+$ is a Morita extension of $T$ to $\Sigma^+$, then $\Pi$ is full.
\end{prop}

\begin{proof}
Let $M$ and $N$ be models of $T^+$ with $h:\Pi(M)\rightarrow\Pi(N)$ an arrow in $\text{Mod}(T)$. This means that $h:M|_\Sigma\rightarrow N|_\Sigma$ is an elementary embedding. We show that the map $h^+:M\rightarrow N$ is an elementary embedding and therefore an arrow in $\text{Mod}(T^+)$. Since $\Pi(h^+)=h$ this will imply that $\Pi$ is full.

Let $\phi(x_1,\ldots, x_n, \overline{x}_1,\ldots, \overline{x}_m)$ be a $\Sigma^+$-formula and let $a_1,\ldots, a_n, \overline{a}_1,\ldots, \overline{a}_m$ be elements of $M$ of the same sorts as the variables $x_1,\ldots, x_n, \overline{x}_1,\ldots, \overline{x}_m$. Lemma 5.1 implies that there is a code $\xi(x_1,\ldots, x_n, y_{11}, \ldots, y_{n2})$ and elements $b_{11},\ldots, b_{n2}$ of $M$ such that $M\vDash\xi[a_1,\ldots, a_n, b_{11}, \ldots, b_{n2}]$. The definition of the map $h^+$ implies that $N\vDash\xi[h^+(a_1,\ldots, a_n, b_{11}, \ldots, b_{n2})]$. We now show that $M\vDash\phi[a_1,\ldots, a_n, \overline{a}_1,\ldots,\overline{a}_m]$ if and only if $N\vDash\phi[h^+(a_1,\ldots, a_n, \overline{a}_1,\ldots,\overline{a}_m)]$. By Theorem 4.3 there is a $\Sigma$-formula $\phi^*(\overline{x}_1,\ldots, \overline{x}_m,y_{11}, \ldots, y_{n2})$ such that 
\begin{equation}
\begin{aligned}
T^+\vDash \forall_{\sigma_1} x_1\ldots&\forall_{\sigma_n} x_n\forall_{\overline{\sigma}_1} \overline{x}_1\ldots\forall_{\overline{\sigma}_m}\overline{x}_m\forall_{\sigma_{11}} y_{11}\ldots\forall_{\sigma_{n2}} y_{n2}\big(\xi(x_1,\ldots, y_{n2})\rightarrow\\
&\big(\phi(x_1,\ldots x_n, \overline{x}_1,\ldots, \overline{x}_m)\leftrightarrow\phi^*(\overline{x}_1,\ldots,\overline{x}_m, y_{11},\ldots, y_{n2})\big)\big)
\end{aligned}
\end{equation}
We then see that the following string of equivalences holds.
\begin{align*}
M\vDash \phi[a_1,\ldots, a_n, \overline{a}_1,\ldots, \overline{a}_m]
\Longleftrightarrow& M\vDash \phi^*[\overline{a}_1,\ldots, \overline{a}_m, b_{11},\ldots, b_{n2}]\\
\Longleftrightarrow& M|_\Sigma\vDash\phi^*[\overline{a}_1,\ldots, \overline{a}_m, b_{11},\ldots, b_{n2}]\\
\Longleftrightarrow& N|_\Sigma\vDash\phi^*[h(\overline{a}_1,\ldots, \overline{a}_m, b_{11},\ldots, b_{n2})]\\
\Longleftrightarrow& N\vDash\phi^*[h(\overline{a}_1,\ldots, \overline{a}_m, b_{11},\ldots, b_{n2})]\\
\Longleftrightarrow& N\vDash\phi^*[h^+(\overline{a}_1,\ldots, \overline{a}_m, b_{11},\ldots, b_{n2})]\\
\Longleftrightarrow& N\vDash\phi[h^+(a_1,\ldots, a_n, \overline{a}_1,\ldots, \overline{a}_m)]
\end{align*}
The first and sixth equivalences hold by (5) and the fact that $M$ and $N$ are models of $T^+$, the second and fourth hold since $\phi^*$ is a $\Sigma$-formula, the third since $h:M|_\Sigma\rightarrow N|_\Sigma$ is an elementary embedding, and the fifth by the definition of $h^+$ and the fact that the elements $\overline{a}_1,\ldots, \overline{a}_m, b_{11}, \ldots, b_{n2}$ have sorts in $\Sigma$.
\end{proof}

These three propositions provide us with the resources to show how categorical equivalence is related to Morita equivalence. Our first result follows as an immediate corollary.

\begin{theorem}
Morita equivalence entails categorical equivalence.
\end{theorem}

\begin{proof}
Suppose that $T_1$ and $T_2$ are Morita equivalent. Then there are theories $T_1^1, \ldots, T_1^n$ and $T_2^1,\ldots, T_2^m$ that satisfy the three conditions in the definition of Morita equivalence. Propositions 5.1, 5.2, and 5.3 imply that the $\Pi$ functors between these theories, represented by the arrows in the following figure, are all equivalences.
\begin{center}
\includegraphics{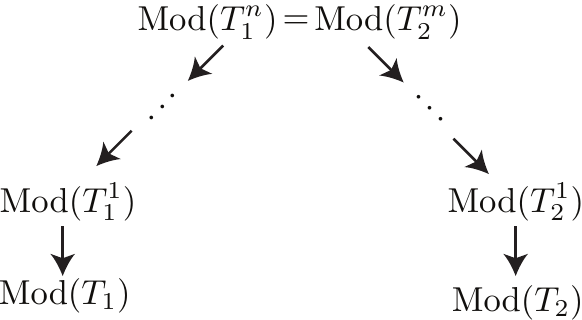}
\end{center}
This implies that $\text{Mod}(T_1)$ is equivalent to $\text{Mod}(T_2)$, and so $T_1$ and $T_2$ are categorically equivalent.
\end{proof}

The converse to Theorem 5.1, however, does not hold. There are theories that are categorically equivalent but not Morita equivalent.\footnote{\cite{halvorson2012} mentions the following example to illustrate a different point.} In order to show this, we need one piece of terminology. A category $C$ is \textbf{discrete} if it is equivalent to a category whose only arrows are identity arrows. 

\begin{theorem}
Categorical equivalence does not entail Morita equivalence.
\end{theorem}

\begin{proof}
Let $\Sigma_1=\{\sigma_1, p_0, p_1, p_2,\ldots\}$ be a signature with a single sort symbol $\sigma_1$ and a countable infinity of predicate symbols $p_i$ of arity $\sigma_1$. Let $\Sigma_2=\{\sigma_2, q_0, q_1, q_2,\ldots\}$ be a signature with a single sort symbol $\sigma_2$ and a countable infinity of predicate symbols $q_i$ of arity $\sigma_2$. Define the $\Sigma_1$-theory $T_1$ and $\Sigma_2$-theory $T_2$ as follows. 
\begin{align*}
T_1&=\{\exists_{\sigma_1=1} x (x=x)\}\\
T_2&=\{\exists_{\sigma_2=1} y (y=y), \forall_{\sigma_2}y(q_0(y)\rightarrow q_1(y)), \forall_{\sigma_2}y(q_0(y)\rightarrow q_2(y)), \ldots\}
\end{align*}
The theory $T_2$ has the sentence $\forall_{\sigma_2}y(q_0(y)\rightarrow q_i(y))$ as an axiom for each $i\in\mathbb{N}$. 

We first show that $T_1$ and $T_2$ are categorically equivalent. It is easy to see that $\text{Mod}(T_1)$ and $\text{Mod}(T_2)$ both have $2^{\aleph_0}$ (non-isomorphic) objects. Furthermore, $\text{Mod}(T_1)$ and $\text{Mod}(T_2)$ are both discrete categories. We show here that $\text{Mod}(T_1)$ is discrete. Suppose that there is an elementary embedding $f:M\rightarrow N$ between models $M$ and $N$ of $T_1$. It must be that $f$ maps the unique element $m\in M$ to the unique element $n\in N$. Furthermore, since $f$ is an elementary embedding, $M\vDash p_i[m]$ if and only if $N\vDash p_i[n]$ for every predicate $p_i\in \Sigma_1$. This implies that $f:M\rightarrow N$ is actually an isomorphism. Every arrow $f:M\rightarrow N$ in $\text{Mod}(T_1)$ is therefore an isomorphism, and there is at most one arrow between any two objects of $\text{Mod}(T_1)$. This immediately implies that $\text{Mod}(T_1)$ is discrete. An analogous argument demonstrates that $\text{Mod}(T_2)$ is discrete. Any bijection between the objects of $\text{Mod}(T_1)$ and $\text{Mod}(T_2)$ is therefore an equivalence of categories.

But $T_1$ and $T_2$ are not Morita equivalent. Suppose for contradiction that $T$ is a ``common Morita extension'' of $T_1$ and $T_2$. Corollary 4.1 implies that there is a $\Sigma_1$-sentence $\phi$ such that $T\vDash\forall_{\sigma_2}yq_0(y)\leftrightarrow\phi$. One can verify using Theorem 4.2 and Corollary 4.1 that the sentence $\phi$ has the following property: If $\psi$ is a $\Sigma_1$-sentence and $T_1\vDash\psi\rightarrow\phi$, then either (i) $T_1\vDash\lnot\psi$ or (ii) $T_1\vDash\phi\rightarrow\psi$. But $\phi$ cannot have this property. Consider the $\Sigma_1$-sentence $$
\psi:=\phi\land\forall_{\sigma_1}x p_i(x)
$$
where $p_i$ is a predicate symbol that does not occur in $\phi$. We trivially see that $T_1\vDash \psi\rightarrow\phi$, but neither (i) nor (ii) hold of $\psi$. This implies that $T_1$ and $T_2$ are not Morita equivalent.
\end{proof}

\section{Conclusion}
We have discussed three formal criteria for theoretical equivalence, and we have shown that they form the following hierarchy.
\begin{center}
\includegraphics{hierarchy.pdf}
\end{center}
This hierarchy yields a precise sense in which definitional equivalence is too strict a criterion for theoretical equivalence. One often has good reason to consider theories with different sort symbols equivalent. But definitional equivalence does not allow one to do this. Morita equivalence, on the other hand, does allow one to capture a sense in which such theories might be equivalent. 

The hierarchy also yields a precise sense in which categorical equivalence is too liberal a criterion for theoretical equivalence. The example from Theorem 5.2 is quite general. Any two theories with discrete categories of models will be categorically equivalent, as long as they have the same number of models. But one often has good reason to consider two such theories inequivalent. For example, there is a sense in which the two theories from Theorem 5.2 do not ``say the same thing.'' According to the theory $T_2$, there is a special predicate $q_0$. If the predicate $q_0$ holds, that completely determines what else is true according to $T_2$. The theory $T_1$, however, singles out no such predicate. If one takes categorical equivalence as the standard for theoretical equivalence, then one is forced to consider $T_1$ and $T_2$ equivalent. Morita equivalence, on the other hand, allows one to consider them inequivalent.

Even though there is a sense in which it is too liberal, categorical equivalence is currently our most promising formal criterion for theoretical equivalence outside the framework of first-order logic. We have seen that it is a weaker criterion than Morita equivalence, but one nonetheless hopes that it is not ``too much weaker.'' One could substantiate this hope by proving a result of the following form.
\begin{quote}
If $T_1$ and $T_2$ are categorically equivalent and $\mathfrak{P}$, then $T_1$ and $T_2$ are Morita equivalent.
\end{quote}
$\mathfrak{P}$ is some additional constraint that $T_1$ and $T_2$ might be required to satisfy. For example, one might hope that the result could be proven when $\mathfrak{P}$ is ``$T_1$ and $T_2$ have finite signatures.'' If a result of this form holds for a general property $\mathfrak{P}$, that would show that categorical equivalence is ``almost as strong'' as Morita equivalence.  

Promising work in this direction has been done by \cite{makkai} and \cite{awodeyforssell}. Makkai shows that if the \textit{ultracategories} $\text{Mod}(T_1)$ and $\text{Mod}(T_2)$ are equivalent, then $T_1$ and $T_2$ are Morita equivalent. Awodey and Forssell
show that if the \textit{topological groupoids} $\text{Mod}(T_1)$ and $\text{Mod}(T_2)$ are equivalent, then $T_1$ and $T_2$ are Morita equivalent. But there is still more work to be done before we completely understand the relationship between Morita equivalence and categorical equivalence.\renewcommand{\thefootnote}{$\star$}\footnote{This material is based upon work supported by the National Science Foundation under Grant No.~DGE 1148900.}


\bibliographystyle{apalike}
\bibliography{/Users/thomaswbarrett/Desktop/Work/masterbib}

\section*{Appendix}
\setcounter{section}{4}
\setcounter{lemma}{1}
\setcounter{theorem}{2}

This appendix contains a proof of Theorem 4.3, which we restate here for convenience.

\begin{theorem}
Let $\Sigma\subset\Sigma^+$ be signatures and $T$ a $\Sigma$-theory. Suppose that $T^+$ is a Morita extension of $T$ to $\Sigma^+$ and that $\phi(x_1,\ldots, x_n, \overline{x}_1,\ldots, \overline{x}_m)$ is a $\Sigma^+$-formula. Then for every code $\xi(x_1,\ldots, y_{n2})$ for the variables $x_1,\ldots, x_n$ there is a $\Sigma$-formula $\phi^*(\overline{x}_1,\ldots, \overline{x}_m, y_{11}, \ldots, y_{n2})$ such that 
\begin{align*}
T^+\vDash \forall_{\sigma_1} x_1\ldots&\forall_{\sigma_n} x_n\forall_{\overline{\sigma}_1} \overline{x}_1\ldots\forall_{\overline{\sigma}_m}\overline{x}_m\forall_{\sigma_{11}} y_{11}\ldots\forall_{\sigma_{n2}} y_{n2}\big(\xi(x_1,\ldots, y_{n2})\rightarrow\\
&\big(\phi(x_1,\ldots x_n, \overline{x}_1,\ldots, \overline{x}_m)\leftrightarrow\phi^*(\overline{x}_1,\ldots,\overline{x}_m, y_{11},\ldots, y_{n2})\big)\big)
\end{align*}
\end{theorem}

We first prove the following lemma. Given a $\Sigma^+$-term $t$, we will again write $t(x_1,\ldots, x_n,\overline{x}_1,\ldots, \overline{x}_m)$ to indicate that the variables $x_1,\ldots, x_n$ have sorts $\sigma_1,\ldots, \sigma_n\in\Sigma^+-\Sigma$ and that the variables $\overline{x}_1,\ldots, \overline{x}_m$ have sorts $\overline{\sigma}_1,\ldots, \overline{\sigma}_m\in\Sigma$.

\begin{lemma}
Let $t(x_1,\ldots, x_n, \overline{x}_1,\ldots, \overline{x}_m)$ be a $\Sigma^+$-term of sort $\sigma$ and $x$ a variable of sort $\sigma$. Let $\xi(x, x_1,\ldots, x_n, y_{11}, \ldots, y_{n2})$ be a code for the variables $x,x_1,\ldots, x_n$. Then there is a $\Sigma$-formula $\phi_t(x,\overline{x}_1,\ldots, \overline{x}_m, y_{11}, \ldots, y_{n2})$ such that
\begin{align*}
T^+\vDash \forall_\sigma x \forall_{\sigma_1} x_1\ldots&\forall_{\sigma_n} x_n\forall_{\overline{\sigma}_1} \overline{x}_1\ldots\forall_{\overline{\sigma}_m}\overline{x}_m\forall_{\sigma_{11}} y_{11}\ldots\forall_{\sigma_{n2}} y_{n2}\big(\xi(x, x_1,\ldots, y_{n2})\rightarrow\\
&\big(t(x_1, \ldots, \overline{x}_m)=x\leftrightarrow\phi_t(x, \overline{x}_1,\ldots,\overline{x}_m, y_{11},\ldots, y_{n2})\big)\big)
\end{align*}
If $\sigma\in\Sigma$, then $x$ will not appear in the code $\xi$. If $\sigma\in\Sigma^+-\Sigma$, then $x$ will not appear in the $\Sigma$-formula $\phi_t$.
\end{lemma}

\begin{proof}
We induct on the complexity of $t$. First, suppose that $t$ is a variable $x_i$ of sort $\sigma$. If $\sigma\in\Sigma$, then there are no variables in $t$ with sorts in $\Sigma^+-\Sigma$. So $\xi$ must be the empty code. Let $\phi_t(x, x_i)$ be the $\Sigma$-formula $x=x_i$. This choice of $\phi_t$ trivially satisfies the desired property. If $\sigma\in\Sigma^+-\Sigma$, then there are four cases to consider. We consider the cases where $\sigma$ is a product sort and a subsort. The coproduct and quotient cases follow analogously. Suppose that $T^+$ defines $\sigma$ as a product sort with projections $\pi_1$ and $\pi_2$ of arity $\sigma\rightarrow\sigma_1$ and $\sigma\rightarrow\sigma_2$. A code $\xi$ for the variables $x$ and $x_i$ must therefore be the formula
$$
\pi_1(x)=y_1\land\pi_2(x)=y_2\land\pi_1(x_i)=y_{i1}\land\pi_2(x_i)=y_{i2}
$$
One defines the $\Sigma$-formula $\phi_t$ to be $y_1=y_{i1}\land y_2=y_{i2}$ and verifies that it satisfies the desired property. On the other hand, suppose that $T^+$ defines $\sigma$ as a subsort with injection $i$ of arity $\sigma\rightarrow\sigma_1$. A code $\xi$ for the variables $x$ and $x_i$ is therefore the formula
$$
i(x)=y\land i(x_i)=y_{i1}
$$
Let $\phi_t$ be the $\Sigma$-formula $y=y_{i1}$. The desired property again holds.

Second, suppose that $t$ is the constant symbol $c$. Note that it must be the case that $c$ is of sort $\sigma\in\Sigma$. If $c\in\Sigma$, then letting $\phi_t$ be the $\Sigma$-formula $x=c$ trivially yields the result. If $c\in\Sigma^+-\Sigma$, then there is some $\Sigma$-formula $\psi(x)$ that $T^+$ uses to explicitly define $c$. Letting $\phi_t=\psi$ yields the desired result.

For the third (and final) step of the induction, we suppose that $t$ is a term of the form
$$
f\big(t_1(x_1,\ldots, x_n, \overline{x}_1,\ldots, \overline{x}_m), \ldots, t_k(x_1,\ldots, x_n, \overline{x}_1,\ldots, \overline{x}_m)\big)
$$
where $f\in\Sigma^+$ is a function symbol. We show that the result holds for $t$ if it holds for all of the terms $t_1,\ldots, t_k$. There are three cases to consider. First, if $f\in\Sigma$, then it must be that $f$ has arity $\sigma_1\times\ldots\times\sigma_k\rightarrow\sigma$, where $\sigma,\sigma_1,\ldots, \sigma_k\in\Sigma$. Let $\xi$ be a code for $x_1,\ldots, x_n$. We define $\phi_t$ to be the $\Sigma$-formula
$$
\exists_{\sigma_1} z_1\ldots\exists_{\sigma_k}z_k\big(\phi_{t_1}(z_1, \overline{x}_1,\ldots, y_{n2})\land\ldots\land\phi_{t_k}(z_k, \overline{x}_1,\ldots y_{n2})\land f(z_1,\ldots, z_k)=x\big)
$$
where each of the $\phi_{t_i}$ exists by our inductive hypothesis. One can verify that $\phi_t$ satisfies the desired property. Second, if $f\in\Sigma^+-\Sigma$ is defined by a $\Sigma$-formula $\psi(z_1,\ldots, z_k, x)$ then one defines $\phi_t$ in an analogous manner to above. (Note that in this case the arity of $f$ is again $\sigma_1\times\ldots\times\sigma_k\rightarrow \sigma$ with $\sigma_1,\ldots, \sigma_k,\sigma\in\Sigma$.) 

Third, we need to verify that the result holds if $f$ is a function symbol that is used in the definition of a new sort. We discuss the cases where $f$ is $\pi_1$ and where $f$ is $\epsilon$. Suppose that $f$ is $\pi_1$ with arity $\sigma\rightarrow\sigma_1$. Then it must be that the term $t_1$ is a variable $x_i$ of sort $\sigma$ since there are no other $\Sigma^+$-terms of sort $\sigma$. So the term $t$ is $\pi_1(x_i)$. Let $\xi(x_i, y_{i1}, y_{i2})$ be a code for $x_i$. It must be that $\xi$ is the formula
$$
\pi_1(x_i)=y_{i1}\land\pi_2(x_i)=y_{i2}
$$
Letting $\phi_t$ be the formula $y_{i1}=x$ yields the desired result. On the other hand, suppose that $f$ is the function symbol $\epsilon$ of arity $\sigma_1\rightarrow\sigma$, where $\sigma$ is a quotient sort defined by the $\Sigma$-formula $\psi(z_1, z_2)$. The term $t$ in this case is $\epsilon(t_1(x_1,\ldots, x_n, \overline{x}_1,\ldots, \overline{x}_m))$ and we assume that the result holds for the $\Sigma^+$-term $t_1$ of sort $\sigma_1\in\Sigma$. Let $\xi$ be a code for the variables $x, x_1,\ldots, x_n$. This code determines a code $\overline{\xi}$ for the variables $x_1,\ldots, x_n$ by ``forgetting'' the conjunct $\epsilon(y)=x$ that involves the variable $x$. We use the code $\overline{\xi}$ and the inductive hypothesis to obtain the formula $\phi_{t_1}$. Then we define $\phi_t$ to be the $\Sigma$-formula
$$
\exists_{\sigma_1} z\big(\phi_{t_1}(z, \overline{x}_1,\ldots, \overline{x}_m, y_{11},\ldots, y_{n2})\land\psi(y, z)\big)
$$
Considering the original code $\xi$, one verifies that the result holds for $\phi_{t_1}$.
\end{proof}

We now turn to the proof of the main result.

\begin{proof}[Proof of Theorem 4.3]
We induct on the complexity of $\phi$. Suppose that $\phi$ is the formula $t(x_1,\ldots, x_n, \overline{x}_1,\ldots, \overline{x}_m)=s(x_1,\ldots, x_n, \overline{x}_1,\ldots,\overline{x}_m)$ for $\Sigma^+$-terms $t$ and $s$. Let $\xi(x_1,\ldots, y_{n2})$ be a code for $x_1,\ldots, x_n$ and let $x$ be a variable of sort $\sigma$. If $t$ and $s$ are both terms of sort $\sigma\in\Sigma$, then one uses Lemma 4.2 and the code $\xi$ to generate the $\Sigma$-formulas $\phi_t(x,\overline{x}_1,\ldots, \overline{x}_m, y_{11}, \ldots, y_{n2})$ and $\phi_s(x,\overline{x}_1,\ldots, \overline{x}_m, y_{11}, \ldots, y_{n2})$. The $\Sigma$-formula $\phi^*$ is then defined to be 
$$
\exists_\sigma x\big(\phi_t(x,\overline{x}_1,\ldots, \overline{x}_m, y_{11}, \ldots, y_{n2})\land\phi_s(x,\overline{x}_1,\ldots, \overline{x}_m, y_{11}, \ldots, y_{n2})\big)
$$
One can verify that this definition of $\phi^*$ satisfies the desired result. 

If $t$ and $s$ are of sort $\sigma\in\Sigma^+-\Sigma$, then there are four cases to consider. We show that the result holds when $T^+$ defines $\sigma$ as a product sort or a quotient sort. The coproduct and subsort cases follow analogously. If $T^+$ defines $\sigma$ as a product sort with projections $\pi_1$ and $\pi_2$ of arity $\sigma\rightarrow\sigma_1$ and $\sigma\rightarrow\sigma_2$, then we define a code $\overline{\xi}(x, x_1,\ldots, y_{n2}, v_1, v_2)$ for the variables $x, x_1,\ldots, x_n$ by
$$
\xi(x_1,\ldots, y_{n2})\land\pi_1(x)=v_1\land\pi_2(x)=v_2
$$
Lemma 4.2 and the code $\overline{\xi}$ for the variables $x,x_1,\ldots, x_n$ generate the $\Sigma$-formulas $\phi_t(\overline{x}_1,\ldots, \overline{x}_m, y_{11}, \ldots, y_{n2}, v_1, v_2)$ and $\phi_s(\overline{x}_1,\ldots, \overline{x}_m, y_{11}, \ldots, y_{n2}, v_1, v_2)$. We then define the $\Sigma$-formula $\phi^*$ to be
\begin{align*}
\exists_{\sigma_1}v_1\exists_{\sigma_2}v_2\big(&\phi_t(\overline{x}_1,\ldots, \overline{x}_m, y_{11}, \ldots, y_{n2}, v_1, v_2)\\ 
&\land \phi_s(\overline{x}_1,\ldots, \overline{x}_m, y_{11}, \ldots, y_{n2}, v_1, v_2)\big)
\end{align*}
One can verify that $\phi^*$ again satisfies the desired result. 

If $T^+$ defines $\sigma$ as a quotient sort with projection $\epsilon$ of arity $\sigma_1\rightarrow\sigma$, then we again define a new code $\overline{\xi}(x, x_1,\ldots, y_{n2}, v)$ for the variables $x, x_1,\ldots, x_n$ by
$$
\xi(x_1,\ldots, y_{n2})\land\epsilon(v)=x
$$
Lemma 4.2 and the code $\overline{\xi}$ for the variables $x, x_1,\ldots, x_n$ again generate the $\Sigma$-formulas $\phi_t(\overline{x}_1,\ldots, \overline{x}_m, y_{11}, \ldots, y_{n2}, v)$ and $\phi_s(\overline{x}_1,\ldots, \overline{x}_m, y_{11}, \ldots, y_{n2}, v)$. We define the $\Sigma$-formula $\phi^*$ to be
$$
\exists_{\sigma_1} v \big(\phi_t(\overline{x}_1,\ldots, \overline{x}_m, y_{11}, \ldots, y_{n2}, v)\land \phi_s(\overline{x}_1,\ldots, \overline{x}_m, y_{11}, \ldots, y_{n2}, v)\big)
$$
One again verifies that this $\phi^*$ satisfies the desired property. So the result holds when $\phi$ is of the form $t=s$ for $\Sigma^+$-terms $t$ and $s$.

Now suppose that $\phi(x_1,\ldots, x_n,\overline{x}_1,\ldots, \overline{x}_m)$ is a $\Sigma^+$-formula of the form 
$$
p(t_1(x_1,\ldots, x_n,\overline{x}_1,\ldots,\overline{x}_m),\ldots, t_k(x_1,\ldots, x_n, \overline{x}_1,\ldots, \overline{x}_m))
$$ 
where $p$ has arity $\sigma_1\times\ldots\times\sigma_k$. Note that it must be that $\sigma_1,\ldots,\sigma_k\in\Sigma$. Either $p\in\Sigma$ or $p\in\Sigma^+-\Sigma$. We consider the second case. (The first is analogous.) Let $\psi(z_1,\ldots, z_k)$ be the $\Sigma$-formula that $T^+$ uses to explicitly define $p$ and let $\xi(x_1,\ldots, y_{n2})$ be a code for $x_1,\ldots, x_n$. Lemma 4.2 and $\xi$ generate the $\Sigma$-formulas $\phi_{t_i}(z_i, \overline{x}_1,\ldots, \overline{x}_m, y_{11}, \ldots, y_{n2})$ for each $i=1,\ldots, k$. We define $\phi^*$ to be the $\Sigma$-formula
\begin{align*}
\exists_{\sigma_1}z_1\ldots\exists_{\sigma_k}z_k\big(&\phi_{t_1}(z_1,\overline{x}_1,\ldots, \overline{x}_m,y_{11},\ldots, y_{n2})\land\ldots\\ &\land\phi_{t_k}(z_k,\overline{x}_1,\ldots, \overline{x}_m,y_{11},\ldots, y_{n2})\land\psi(z_1,\ldots, z_k)\big)
\end{align*}
One can again verify that the result holds for this choice of $\phi^*$.

We have covered the ``base cases'' for our induction. We now turn to the inductive step. We consider the cases of $\lnot, \land$, and $\forall$. Suppose that the result holds for $\Sigma^+$-formulas $\phi_1$ and $\phi_2$. Then it trivially holds for $\lnot\phi_1$ by letting $(\lnot\phi)^*$ be $\lnot(\phi^*)$. It also trivially holds for $\phi_1\land\phi_2$ by letting $(\phi_1\land\phi_2)^*$ be $\phi_1^*\land\phi_2^*$. 

The $\forall_{\sigma_i}$ case requires more work. If $x_i$ is a variable of sort $\sigma_i\in\Sigma$, we let $(\forall_{\sigma_i} x_i \phi_1)^*$ be $\forall_{\sigma_i} x_i (\phi_1^*)$. The only non-trivial part of the inductive step is when one quantifies over variables with sorts in $\Sigma^+-\Sigma$. Suppose that $\phi(x_1,\ldots, x_n, \overline{x}_1,\ldots, \overline{x}_m)$ is a $\Sigma^+$-formula and that the result holds for it. We let $x_i$ be a variable of sort $\sigma_i\in\Sigma^+-\Sigma$ and we show that the result also holds for the $\Sigma$-formula $\forall_{\sigma_i}x_i\phi(x_1,\ldots, x_n, \overline{x}_1,\ldots, \overline{x}_m)$. There are again four cases. We show that the result holds when $\sigma_i$ is a product sort and a coproduct sort. The cases of subsorts and quotient sorts follow analogously.

Suppose that $T^+$ defines $\sigma_i$ as a product sort with projections $\pi_1$ and $\pi_2$ of arity $\sigma_i\rightarrow\sigma_{i1}$ and $\sigma_i\rightarrow\sigma_{i2}$. Let $\xi(x_1,\ldots, y_{n2})$ be a code for the variables $x_1,\ldots, x_{i-1}, x_{i+1}, \ldots, x_n$ (these are all of the free variables in $\forall_{\sigma_i}x_i\phi$ with sorts in $\Sigma^+-\Sigma$). We define a code $\overline{\xi}$ for the variables $x_1,\ldots, x_{i-1}, x_i, x_{i+1}, \ldots, x_n$ by
$$
\xi(x_1,\ldots, y_{n2})\land\pi_1(x_i)=v_1\land\pi_2(x_i)=v_2
$$
One uses the code $\overline{\xi}$ and the inductive hypothesis to generate the $\Sigma$-formula $\phi^*(\overline{x}_1,\ldots, \overline{x}_m, y_{11},\ldots, y_{n2}, v_1, v_2)$. We then define the $\Sigma$-formula $(\forall_{\sigma_i}x_i\phi)^*$ to be
$$
\forall_{\sigma_{i1}} v_1\forall_{\sigma_{i2}}v_2\phi^*(\overline{x}_1,\ldots, \overline{x}_m, y_{11},\ldots, y_{n2}, v_1, v_2)
$$
And one verifies that the desired result holds for this choice of $(\forall_{\sigma_i}x_i\phi)^*$. (The definition of $(\forall_{\sigma_i}x_i\phi)^*$ is perfectly intuitive. Quantifying over a variable $x_i$ of product sort $\sigma_i$ can be thought of as ``quantifying over pairs of elements of sorts $\sigma_{i1}$ and $\sigma_{i2}$.'')

Suppose that $T^+$ defines $\sigma_i$ as a coproduct sort with injections $\rho_1$ and $\rho_2$ of arity $\sigma_{i1}\rightarrow\sigma_{i}$ and $\sigma_{i2}\rightarrow\sigma_i$. Let $\xi(x_1,\ldots, y_{n2})$ be a code for the variables $x_1,\ldots, x_{i-1}, x_{i+1}, \ldots, x_n$ (these are again all of the free variables in $\forall_{\sigma_i}x_i\phi$ with sorts in $\Sigma^+-\Sigma$). We define two different codes $\overline{\xi}$ for the variables $x_1,\ldots, x_{i-1}, x_i, x_{i+1}, \ldots, x_n$ by
\begin{gather*}
\xi(x_1,\ldots, y_{n2})\land\rho_1(v_1)=x_i\\
\xi(x_1,\ldots, y_{n2})\land\rho_2(v_2)=x_i
\end{gather*}
We will call the first code $\xi'(x_1,\ldots, y_{n2}, v_1)$ and the second $\xi''(x_1,\ldots, y_{n2}, v_2)$. We use these two codes and the inductive hypothesis to generate $\Sigma$-formulas $\phi^{*'}$ and $\phi^{*''}$. We then define the $\Sigma$-formula $(\forall_{\sigma_i}x_i\phi)^*$ to be
\begin{align*}
\forall_{\sigma_{i1}}v_1\forall_{\sigma_{i2}}v_2\big(&\phi^{*'}(\overline{x}_1,\ldots,\overline{x}_m,y_{11}, \ldots, y_{n2}, v_2)\\
&\land\phi^{*''}(\overline{x}_1,\ldots, \overline{x}_m,y_{11},\ldots, y_{n2}, v_2)\big)
\end{align*}
One can verify that the desired result holds again for this definition of $(\forall_{\sigma_i}x_i\phi)^*$. (The definition is again intuitive. Quantifying over a variable $x_i$ of coproduct sort $\sigma_i$ can be thought of as ``quantifying over \textit{both} elements of sort $\sigma_{i1}$ and elements of sort $\sigma_{i2}$.'')
\end{proof}

\end{document}